\documentclass[12pt]{amsart}
\usepackage{amsmath,amsthm,amsfonts,amscd,amssymb,eucal,latexsym,mathrsfs, appendix}
\usepackage[numbers,sort&compress]{natbib}
\usepackage[all,cmtip]{xy}

\newtheorem{theorem}{Theorem}[section]

\newtheorem{lemma}[theorem]{Lemma}
\newtheorem{proposition}[theorem]{Proposition}

\theoremstyle{definition}
\newtheorem{definition}[theorem]{Definition}
\newtheorem{remark}[theorem]{Remark}

\newtheorem{question}[theorem]{Question}

\newcommand{\mdim}{{\rm mdim}}
\newcommand{\ord}{{\rm ord}}

\newcommand{\cB}{{\mathcal B}}

\newcommand{\cW}{{\mathcal W}}
\newcommand{\cD}{{\mathcal D}}

\newcommand{\cC}{{\mathcal C}}
\newcommand{\cU}{{\mathcal U}}

\newcommand{\cQ}{{\mathcal Q}}

\newcommand{\cR}{{\mathcal R}}
\newcommand{\cS}{{\mathcal S}}
\newcommand{\cZ}{{\mathcal Z}}

\newcommand{\cV}{{\mathcal V}}

\newcommand{\Zb}{{\mathbb Z}}

\newcommand{\Rb}{{\mathbb R}}
\newcommand{\Nb}{{\mathbb N}}
\newcommand{\sB}{{\mathscr B}}
\newcommand{\sA}{{\mathscr A}}
\newcommand{\sD}{{\mathscr D}}
\newcommand{\sE}{{\mathscr E}}
\newcommand{\sF}{{\mathscr F}}

\newcommand{\diam}{{\rm diam}}

\newcommand{\rM}{{\rm M}}

\newcommand{\Sym}{{\rm Sym}}
\newcommand{\Map}{{\rm Map}}
\newcommand{\mesh}{{\rm mesh}}

\allowdisplaybreaks

\begin{document}

\title{Sofic Mean Dimension}

\author{Hanfeng Li}
\thanks{Partially supported by NSF Grants DMS-0701414 and DMS-1001625.}

\address{\hskip-\parindent
Department of Mathematics, Chongqing University,
Chongqing 401331, China.
Department of Mathematics, SUNY at Buffalo,
Buffalo NY 14260-2900, U.S.A.}
\email{hfli@math.buffalo.edu}


\date{May 8, 2013}

\begin{abstract}
We introduce mean dimensions for continuous actions of countable sofic groups on compact metrizable spaces. These generalize the Gromov-Lindenstrauss-Weiss mean dimensions for actions of countable amenable groups, and are useful for distinguishing continuous actions of countable sofic groups with infinite entropy.
\end{abstract}

\maketitle

\section{Introduction} \label{S-introduction}

Mean dimension was introduced by Gromov \cite{Gromov} about a decade ago, as an analogue of dimension for dynamical systems, and was studied systematically
by Lindenstrauss and Weiss \cite{LW} for continuous actions of countable amenable groups on compact metrizable spaces.
Among many beautiful results they obtained, it is especially notable that they used mean dimension to show that there exits a minimal action of $\Zb$ on some compact metrizable space
which can not be equivariantly embedded into $[0, 1]^\Zb$ equipped with the shift $\Zb$-action.
Mean dimension is further explored in \cite{Coornaert, CK, Gutman, Krieger06,  Krieger09, Lindenstrauss}.

The notion of sofic groups was also introduced by Gromov \cite{Gromov99} around the same time.
The class of sofic groups include all discrete amenable groups and residually finite groups, and it is still an open question whether every group
is sofic. For some nice exposition on sofic groups, see \cite{CC, ES05, ES06, Pestov, Weiss00}.
Using the idea of counting sofic approximations, in \cite{Bowen10} Bowen
defined entropy for measure-preserving actions of countable sofic groups on probability measure spaces,
when there exists a countable generating partition with finite Shannon entropy. Together with David Kerr, in \cite{KL11, KerLi10amenable} we extended Bowen's measure entropy to all measure-preserving actions of countable sofic groups on standard probability measure spaces, and defined topological entropy for
continuous actions of countable sofic groups on compact metrizable spaces. The sofic measure entropy and sofic topological entropy are related by the variational principle \cite{KL11}. Furthermore, the sofic entropies coincide with the classical entropies when the group is amenable \cite{Bowen10a, KerLi10amenable}.

The goal of this article is to extend mean dimension to continuous actions of countable sofic groups $G$ on compact metrizable spaces $X$. In order to define
sofic mean dimension,
we use some approximate actions of $G$ on finite sets as models, and
replace $X$ by certain spaces of approximately $G$-equivariant maps from the finite sets to $X$, which appeared first in the definition of sofic topological entropy \cite{KerLi10amenable}. A novelty here is that we replace open covers of $X$ by certain open covers on these map spaces.

Lindenstrauss and Weiss studied two kinds of mean dimensions for actions of countable amenable groups in \cite{LW}, one is topological, as the analogue of the covering dimension, and the other is metric, as the analogue of the lower box dimension.

We define sofic mean topological dimension and establish some basic properties in Section~\ref{S-sofic mean topological dim}, and show that it coincides with the Gromov-Lindenstrauss-Weiss mean topological dimension when the group is amenable in Section~\ref{S-top amenable}. Similarly, we discuss sofic metric mean dimension in Sections~\ref{S-sofic metric mean dim} and \ref{S-metric amenable}. It is shown in Section~\ref{S-comparison} that  sofic mean topological dimension is always bounded above by  sofic metric mean dimension. We calculate sofic mean dimensions for some Bernoulli shifts and show that every non-trivial factor of the shift action of $G$ on $[0, 1]^G$ has positive sofic mean dimensions in Section~\ref{S-shifts}. In the last section, we show that
actions with small-boundary property have zero or $-\infty$ sofic mean topological dimension.

To round up this section, we fix some notation.

\begin{definition} \label{D-sofic group}
For $d\in \Nb$, we write $[d]$ for the set $\{1, \dots, d\}$ and
 $\Sym(d)$ for the permutation group
of $[d]$. A countable group $G$ is called {\it sofic} if there is a {\it sofic approximation sequence} $\Sigma= \{ \sigma_i : G \to \Sym (d_i ) \}_{i=1}^\infty$ for $G$, namely the following three conditions are satisfied:

\begin{enumerate}
\item for any $s, t\in G$, one has $\lim_{i\to \infty}\frac{|\{a\in [d_i]: \sigma_i(s)\sigma_i(t)(a)=\sigma_i(st)(a)\}|}{d_i}=1$;

\item for any distinct $s, t\in G$, one has $\lim_{i\to \infty}\frac{|\{a\in [d_i]: \sigma_i(s)(a)=\sigma_i(t)(a)\}|}{d_i}=0$;

\item $\lim_{i\to \infty}d_i=+\infty$.
\end{enumerate}
For a map $\sigma$ from $G$ to $\Sym(d)$ for some $d\in \Nb$, we write $\sigma(s)(a)$ as $\sigma_s(a)$ or $sa$, when there is no confusion.
We say that $\sigma$ is a {\it good enough sofic approximation} for $G$ if for some large finite subset $F$ of $G$ which will be clear from the context,
one has $\frac{|\{a\in [d]: \sigma(s)\sigma(t)(a)=\sigma(st)(a)\}|}{d}$ very close to $1$ for all  $s, t \in F$ and $\frac{|\{a\in [d]: \sigma(s)(a)=\sigma(t)(a)\}|}{d}$ very close to $0$ for all distinct $s, t\in F$.
\end{definition}

Throughout this paper, $G$ will be a countable sofic group with identity element $e_G$, and we fix a sofic approximation sequence $\Sigma$ for $G$.

{\it Acknowledgements.}  I thank Yonatan Gutman and David Kerr for very helpful discussions, and  Lewis Bowen and Elon Lindenstrauss for comments. I am grateful to the referee for extremely helpful comments, which improved the paper greatly.

\section{Sofic Mean Topological Dimension} \label{S-sofic mean topological dim}

In this section we define the sofic mean topological dimension and establish some basic properties.

We start with recalling the definitions of covering dimension of compact metrizable spaces and
mean topological dimension for actions of countable amenable groups.
For a compact space $Y$ and two finite open covers $\cU$ and $\cV$ of $Y$, we say that $\cV$ {\it refines} $\cU$, and write $\cV\succ \cU$, if every element of $\cV$ is contained in some element of $\cU$.

\begin{definition} \label{D-order}
Let $Y$ be a compact space and $\cU$ a finite open cover of $Y$. We denote
$$\ord(\cU)=\max_{y\in Y}\sum_{U\in \cU}1_U(y)-1,  \mbox{ and } \cD(\cU)=\min_{\cV\succ \cU}\ord(\cV),$$
where $\cV$ ranges over finite open covers of $Y$ refining $\cU$.
\end{definition}

For a compact metrizable space $X$, its {\it (covering) dimension} $\dim(X)$ is defined as
$\sup_{\cU}\cD(\cU)$ for $\cU$ ranging over finite open covers of $X$.

\begin{definition} \label{D-amenable}
A countable group $G$ is called {\it amenable} if for any finite subset $K$ of $G$ and any $\varepsilon>0$  there exists a finite subset $F$ of $G$ with $|KF\setminus F|<\varepsilon |F|$. Equivalently, $G$ has  a {\it left F{\o}lner sequence} $\{F_n\}_{n\in \Nb}$, i.e. each $F_n$ is a nonempty finite subset of $G$ and $\frac{|sF_n\setminus F_n|}{|F_n|}\to 0$ as $n\to 0$  for every $s\in G$.
\end{definition}

Let a  countable amenable group $G$ act continuously on a compact metrizable space $X$. Let $\cU$ be a finite open  cover of $X$.
For a nonempty finite subset $F$ of $G$, we set $\cU^F=\bigvee_{s\in F}s^{-1}\cU$. The function $F\mapsto \cD(\cU^F)$ defined on the set of nonempty finite subsets of $G$ satisfies
the conditions of the Ornstein-Weiss lemma \cite{OW} \cite[Theorem 6.1]{LW}, thus $\frac{\cD(\cU^F)}{|F|}$ converges to some real number, denoted by $\mdim(\cU)$, as $F$ becomes more and more left invariant. That is, for any $\varepsilon>0$, there exist a nonempty finite subset $K$ of $G$ and $\delta>0$ such that $|\frac{\cD(\cU^F)}{|F|}-\mdim(\cU)|<\varepsilon$ for every nonempty finite subset $F$ of $G$ satisfying $|KF\setminus F|<\delta |F|$.
In terms of any left F{\o}lner sequence $\{F_n\}_{n\in \Nb}$ of $G$, one has
$$ \mdim(\cU)=\lim_{n\to \infty}\frac{\cD(\cU^{F_n})}{|F_n|}.$$
The {\it mean topological dimension} of $X$ \cite[page 13]{LW} is defined as
$$\mdim(X)=\sup_{\cU}\mdim(\cU),$$
where $\cU$ ranges over finite open covers of $X$.

Throughout the rest of this section, we fix a countable sofic group $G$ and a sofic approximation sequence $\Sigma = \{ \sigma_i : G \to \Sym (d_i ) \}_{i=1}^\infty$ for $G$, as defined in Section~\ref{S-introduction}.
Let $\alpha$ be a continuous action of $G$ on a compact metrizable space $X$.

Let $\rho$ be a continuous pseudometric on $X$.
For a given $d\in\Nb$, we define on the set of all maps from $[d]$ to $X$ the pseudometrics
\begin{align*}
\rho_2 (\varphi , \psi ) &= \bigg( \frac{1}{d} \sum_{a\in [d]} (\rho (\varphi (a),\psi (a)))^2 \bigg)^{1/2} , \\
\rho_\infty (\varphi ,\psi ) &= \max_{a\in [d]} \rho (\varphi (a),\psi (a)) .
\end{align*}

\begin{definition}\label{D-map top}
Let $F$ be a nonempty finite subset of $G$ and $\delta > 0$.
Let $\sigma$ be a map from $G$ to $\Sym (d)$ for some $d\in\Nb$.
We define $\Map (\rho ,F,\delta ,\sigma )$ to be the set of all maps $\varphi : [d] \to X$ such that
$\rho_2 (\varphi\circ\sigma_s , \alpha_s \circ\varphi ) \le \delta$ for all $s\in F$. We consider $\Map (\rho ,F,\delta ,\sigma )$ to be a topological space
with the topology inherited from $X^d$.
\end{definition}

The space $\Map (\rho ,F,\delta ,\sigma )$ appeared first in \cite[Section 2]{KerLi10amenable}, and was used to define the topological entropy of the action $\alpha$. Eventually we shall take $\sigma$ to be $\sigma_i$ for large $i$. Then the condition (1) in the definition of $\Sigma$ says that $\sigma$ is approximately a group homomorphism of $G$ into $\Sym(d)$, and therefore we can think of $\sigma$  as an approximate action of $G$ on $[d]$. The space
$\Map (\rho ,F,\delta ,\sigma )$ is the set of approximately $G$-equivariant maps from $[d]$ into $X$.

For a finite open cover $\cU$ of $X$, we denote by $\cU^d$ the finite open cover of $X^{[d]}$ consisting of $U_1\times U_2\times \cdots \times U_d$ for $U_1, \dots, U_d\in \cU$.
Note that $\Map(\rho, F, \delta, \sigma)$ is a closed subset of $X^{[d]}$. Consider the restriction
$\cU^d|_{\Map(\rho, F, \delta, \sigma)}=\cU^d\cap \Map(\rho, F, \delta, \sigma)$ of $\cU^d$ to $\Map(\rho, F, \delta, \sigma)$.
Denote $\cD(\cU^d|_{\Map(\rho, F, \delta, \sigma)})$ by $\cD(\cU, \rho, F, \delta, \sigma)$.

The set $[d]$ is the analogue of an approximately left invariant finite subset $H$ of $G$ in the amenable group case, $\Map(\rho, F, \delta, \sigma)$ is the analogue of the subset $\{(sx)_{s\in H}: x\in X\}$ of $X^H$ which can be identified with $X$ naturally (this will be made clear in the proof of Theorem~\ref{T-top mean dim} below), $\cU^d|_{\Map(\rho, F, \delta, \sigma)}$ is then the analogue of $\cU^H$,
and $\cD(\cU, \rho, F, \delta, \sigma)$ is the analogue of $\cD(\cU^H)$.

\begin{definition} \label{D-sofic top mean dim}
Let $\rho$ be a compatible metric on $X$.
Let $F$ be a nonempty finite subset of $G$ and
$\delta > 0$. For a finite open cover $\cU$ of $X$ we define
\begin{align*}
\cD_\Sigma(\cU, \rho ,F, \delta ) &=
\varlimsup_{i\to\infty} \frac{\cD(\cU, \rho, F, \delta, \sigma_i)}{d_i},\\
\cD_\Sigma(\cU, \rho ,F ) &= \inf_{\delta > 0} \cD_\Sigma(\cU, \rho ,F, \delta ),\\
\cD_\Sigma(\cU, \rho)&= \inf_{F} \cD_\Sigma(\cU, \rho ,F),
\end{align*}
where $F$ in the third line ranges over the nonempty finite subsets of $G$.
If $\Map (\rho ,F,\delta ,\sigma_i )$ is empty for all sufficiently large $i$, we set
$\cD_\Sigma(\cU, \rho ,F, \delta ) = -\infty$.
We define the {\it sofic mean topological dimension} of $\alpha$  as
$$ \mdim_\Sigma (X, \rho ) = \sup_{\cU} \cD_\Sigma(\cU, \rho)$$
for $\cU$ ranging over  finite open covers of $X$.
As shown by Lemma~\ref{L-sofic top mean dim independent of metric} below, the quantities
$\cD_\Sigma(\cU, \rho ,F ), \cD_\Sigma(\cU, \rho)$ and $\mdim_\Sigma(X, \rho)$ do not depend on the choice of $\rho$, and we shall write them as
$\cD_\Sigma(\cU, F ), \cD_\Sigma(\cU)$ and $\mdim_\Sigma(X)$ respectively. In particular, $\mdim_\Sigma(\cdot)$ is an invariant of topological dynamical systems.
\end{definition}

\begin{remark} \label{R-mean top dim1}
Note that $\cD_\Sigma(\cU, \rho ,F, \delta )$ decreases when $\delta$ decreases and $F$ increases. Thus in the definitions of $\cD_\Sigma(\cU, \rho ,F )$ and $\cD_\Sigma(\cU, \rho)$ one can also replace $\inf_{\delta>0}$ and $\inf_F$ by $\lim_{\delta\to 0}$ and $\lim_{F\to \infty}$ respectively, where $F_1\le F_2$ means  $F_1\subseteq F_2$. If we partially order the set of all such $(F, \delta)$ by $(F, \delta)\ge (F', \delta)$ when $F\supseteq F'$ and $\delta\le \delta'$,
then
$$ \cD_\Sigma(\cU, \rho)= \lim_{(F, \delta)\to \infty} \cD_\Sigma(\cU, \rho ,F, \delta).$$
\end{remark}

\begin{remark} \label{R-mean top dim2}
From Definition~\ref{D-sofic top mean dim} and \cite[Proposition 2.4]{KerLi10amenable} one gets that the following conditions are equivalent:
\begin{enumerate}
\item $\mdim_\Sigma(X)\ge 0$;

\item For any finite subset $F$ of $G$, any $\delta>0$, and any $N\in \Nb$, there is some $i\ge N$ such that $\Map(\rho, F, \delta, \sigma_i)$ is nonempty;

\item The sofic topological entropy $h_\Sigma(X)\ge 0$.
\end{enumerate}
Also, by the variational principle \cite[Theorem 6.1]{KL11}, these conditions imply the following
\begin{enumerate}
\item[(4)] $X$ has a $G$-invariant Borel probability measure.
\end{enumerate}
Every non-amenable countable group has a continuous affine action on a compact metrizable convex set (in some locally convex Hausdorff topological vector space) admitting no fixed point, equivalently, admitting no invariant Borel probability measure \cite[Corollary 4.10.2]{CC}. Thus every non-amenable sofic group has a continuous action on some compact metrizable space with mean topological dimension $-\infty$. However, we do not know whether there is any such an example with a unique invariant Borel probability measure.
\end{remark}

\begin{remark} \label{R-sequence}
The definitions of $\mdim_\Sigma(X), \mdim_{\Sigma, \rM}(X, \rho)$ in the sequel and $\mdim_{\Sigma, \rM}(X)$ depend on the choice of the sofic approximation sequence $\Sigma$.
However, we do not know any examples for which different choices of $\Sigma$ lead to different values of these invariants, though Lewis Bowen \cite{Bowen09} showed that
the sofic measure entropy of the trivial action of ${\rm SL}(n, \Zb)$ for $n\ge 2$ (more generally groups with property ($\tau$)) on the two-point set equipped with the uniform distribution does depend on the choice of $\Sigma$.
\end{remark}

We need the following simple observation several times.

\begin{lemma} \label{L-almost}
Let $\rho$ be a continuous pseudometric on $X$, $F$ a nonempty finite subset of $G$, $\delta>0$, and $\sigma$ a map from $G$ to $\Sym(d)$ for some $d\in \Nb$.
For any $\varphi\in \Map(\rho, F, \delta, \sigma)$ and $s\in F$, setting $\cW=\{a\in [d]: \rho(s\varphi(a), \varphi(sa))\le \sqrt{\delta}\}$, one has
$|\cW|\ge (1-\delta)d$.
\end{lemma}
\begin{proof} This follows from
$$ \delta^2\ge (\rho_2(\alpha_s\circ \varphi, \varphi\circ \sigma_s))^2\ge \frac{1}{d}|[d]\setminus \cW|\delta=(1-\frac{|\cW|}{d})\delta.$$
\end{proof}

\begin{lemma} \label{L-sofic top mean dim independent of metric}
Let $\rho$ and $\rho'$ be compatible metrics on $X$. For any nonempty finite subset $F$ of $G$ and any finite open cover $\cU$ of $X$, one has
$\cD_\Sigma(\cU, \rho ,F )=\cD_\Sigma(\cU, \rho' ,F )$.
\end{lemma}
\begin{proof} By symmetry it suffices to show $\cD_\Sigma(\cU, \rho ,F )\le \cD_\Sigma(\cU, \rho' ,F )$. Let $\delta>0$.
Take $\delta'>0$ be a small positive number which we shall determine in a moment. We claim that for any map $\sigma$ from $G$ to $\Sym(d)$ for some
$d\in \Nb$
one has $\Map(\rho, F, \delta', \sigma)\subseteq \Map(\rho', F, \delta, \sigma)$.
Let $\varphi\in \Map(\rho, F, \delta', \sigma)$. For each $s\in F$, set
\[ \cW_s=\{a\in [d]: \rho(s\varphi(a), \varphi(sa))\le \sqrt{\delta'}\}.\]
By Lemma~\ref{L-almost} one has $|\cW_s|\ge (1-\delta')d$. Taking $\delta'$ small enough, we may assume that
for any $x, y\in X$ with $\rho(x, y)\le \sqrt{\delta'}$, one has $\rho'(x, y)\le \delta/2$. Then one has
\begin{align*}
(\rho'_2(\alpha_s\circ \varphi, \varphi\circ \sigma_s))^2&\le \frac{|\cW_s|}{d}\cdot \frac{\delta^2}{4}+(1-\frac{|\cW_s|}{d})(\diam(X, \rho'))^2\\
&\le \frac{\delta^2}{4}+\delta'(\diam(X, \rho'))^2\le \delta^2,
\end{align*}
granted that $\delta'$ is small enough. Therefore $\varphi\in \Map(\rho', F, \delta, \sigma)$. This proves the claim.

Since $\Map(\rho, F, \delta', \sigma)\subseteq \Map(\rho', F, \delta, \sigma)$, clearly $\cD(\cU, \rho, F, \delta', \sigma)\le \cD(\cU, \rho', F, \delta, \sigma)$. Thus $\cD(\cU, \rho, F)\le \cD(\cU, \rho, F, \delta')\le \cD(\cU, \rho', F, \delta)$. Letting $\delta\to 0$, we get $\cD(\cU, \rho, F)\le \cD(\cU, \rho', F)$ as desired.
\end{proof}

We need the following lemma several times.

\begin{lemma} \label{L-factor1}
Let $\alpha^X$ and $\alpha^Y$ be  continuous actions of $G$ on compact metrizable spaces $X$ and $Y$ respectively, and $\pi: X\rightarrow Y$ be an equivariant continuous map. Let $\rho^X$ and $\rho^Y$ be compatible metrics on $X$ and $Y$ respectively. Let $F$ be a nonempty finite subset of $G$ and $\delta>0$. Then there exists $\delta'>0$ such that for every map $\sigma$ from $G$ to $\Sym(d)$ for some $d\in \Nb$ and every
$\varphi\in \Map(\rho^X, F, \delta', \sigma)$, one has $\pi\circ \varphi\in \Map(\rho^Y, F, \delta, \sigma)$.
\end{lemma}
\begin{proof} Since $\pi$ is continuous and $X$ is compact, we can find $\delta'>0$ small enough such that $|F|\delta'(\diam(Y, \rho^Y))^2\le \delta^2/2$ and for any $x, x'\in X$ with $\rho^X(x, x')\le \sqrt{\delta'}$ one has $\rho^Y(\pi(x), \pi(x'))\le \delta/2$.

Denote by $\cW$ the set of all $a\in [d]$ satisfying $\rho^X(s\varphi(a),\varphi(sa))\le \sqrt{\delta'}$. By Lemma~\ref{L-almost} one has
$|\cW|\ge (1-|F|\delta')d$. For each $a\in \cW$, by the choice of $\delta'$ one has
$$\rho^Y(s\pi(\varphi(a)), \pi(\varphi(sa))=\rho^Y(\pi(s\varphi(a)), \pi(\varphi(sa))\le \delta/2.$$
Thus
\begin{align*}
(\rho^Y_2(\alpha^Y_s\circ \pi\circ \varphi, \pi\circ \varphi\circ \sigma_s))^2&\le \frac{|\cW|}{d}(\frac{\delta}{2})^2+\frac{|[d]\setminus \cW|}{d}(\diam(Y, \rho^Y))^2\\
&\le \delta^2/4+|F|\delta'(\diam(Y, \rho^Y))^2\le \delta^2/4+\delta^2/2<\delta^2,
\end{align*}
and hence $\pi\circ \varphi\in \Map(\rho^Y, F, \delta, \sigma)$.
\end{proof}

Lindenstrauss and Weiss established the next two propositions in the case $G$ is amenable \cite[page 5, Proposition 2.8]{LW}.

\begin{proposition} \label{P-subspace top}
Let $G$ act continuously on a compact metrizable space $X$. Let $Y$ be a closed $G$-invariant subset of $X$. Then $\mdim_\Sigma(Y)\le \mdim_\Sigma(X)$.
\end{proposition}
\begin{proof}
Let $\rho$ be a compatible metric on $X$. Then $\rho$ restricts to a compatible metric $\rho'$ on $Y$.

Let $\cU$ be a finite open cover of $Y$. Then we can find a finite open cover $\cV$ of $X$ such that $\cU$ is the restriction of $\cV$ to $Y$.
Note that $\Map(\rho', F, \delta, \sigma)\subseteq \Map(\rho, F, \delta, \sigma)$ for any nonempty finite subset $F$ of $G$, any $\delta>0$, and any
map $\sigma$ from $G$ to $\Sym(d)$ for some $d\in \Nb$. Furthermore, the restriction of $\cV^d|_{\Map(\rho, F, \delta, \sigma)}$ on $\Map(\rho', F, \delta, \sigma)$ is exactly $\cU^d|_{\Map(\rho', F, \delta, \sigma)}$. Thus $\cD(\cU, \rho', F, \delta, \sigma)\le \cD(\cV, \rho, F, \delta, \sigma)$. It follows that
$\mdim_\Sigma(\cU)\le \mdim_\Sigma(\cV)\le \mdim_\Sigma(X)$. Since $\cU$ is an arbitrary finite open cover of $Y$, we get $\mdim_\Sigma(Y)\le \mdim_\Sigma(X)$.
\end{proof}

\begin{proposition} \label{P-product}
Let $G$ act continuously on a compact metrizable space $X_n$ for each $1\le n< R$, where $R\in  \Nb\cup \{\infty\}$. Consider the product action of $G$ on $X:=\prod_{1\le n<R}X_n$.
Then $\mdim_\Sigma(X)\le \sum_{1\le n<R}\mdim_\Sigma(X_n)$.
\end{proposition}
\begin{proof} Let $\rho$ and $\rho^{(n)}$ be compatible metrics on $X$ and $X_n$ respectively.
Denote by $\pi_n$ the projection of $X$ onto $X_n$. Let $\cU$ be a finite open cover of $X$. Then there are an $N\in \Nb$ with $N<R$ and a finite open cover
$\cV_n$ of $X_n$ for all $1\le n\le N$ such that
$$\cV:=\bigvee^N_{n=1}\pi_n^{-1}(\cV_n)\succ \cU.$$

Let $F$ be a nonempty finite subset of $G$ and $\delta>0$.
By Lemma~\ref{L-factor1} we can find $\delta'>0$ such that for any map $\sigma$ from $G$ to $\Sym(d)$ for some $d\in \Nb$ and any $\varphi\in \Map(\rho, F, \delta', \sigma)$
one has $\pi_n\circ \varphi\in \Map(\rho^{(n)}, F, \delta, \sigma)$ for all $1\le n\le N$. It follows that we have a continuous map
$\Phi_n: \Map(\rho, F, \delta', \sigma)\rightarrow \Map(\rho^{(n)}, F, \delta, \sigma)$ sending $\varphi$ to $\pi_n\circ \varphi$ for each $1\le n\le N$.
Note that
$$\cV^d|_{\Map(\rho, F, \delta', \sigma)}=\bigvee_{n=1}^N\Phi_n^{-1}(\cV_n^d|_{\Map(\rho^{(n)}, F, \delta, \sigma)}).$$
For any finite open covers $\cU_1$ and $\cU_2$ of a compact metrizable space $Y$ one has $\cD(\cU_1\vee \cU_2)\le \cD(\cU_1)+\cD(\cU_2)$ \cite[Corollary 2.5]{LW}.
Thus
\begin{align*}
\cD(\cU, \rho, F, \delta', \sigma)\le \cD(\cV, \rho, F, \delta', \sigma)\le \sum_{n=1}^N \cD(\cV_n, \rho^{(n)}, F, \delta, \sigma),
\end{align*}
and hence $\cD_\Sigma(\cU, \rho)\le \cD_\Sigma(\cU, \rho, F, \delta')\le \sum_{n=1}^N\cD_\Sigma(\cV_n, \rho^{(n)}, F, \delta)$.
Since $F$ and $\delta$ are arbitrary, we get
$$\cD_\Sigma(\cU, \rho)\le \sum_{n=1}^N\cD_\Sigma(\cV_n, \rho^{(n)})\le \sum_{n=1}^N\mdim_\Sigma(X_n)\le \sum_{1\le n <R}\mdim_\Sigma(X_n).$$
Therefore $\mdim_\Sigma(X)\le  \sum_{1\le n<R}\mdim_\Sigma(X_n)$ as desired.
\end{proof}

If a property P for continuous $G$-actions on compact Hausdorff spaces is preserved by products, subsystems, and isomorphisms, and the trivial action of $G$ on the one-point set $\bullet$ has property P, then for any continuous
$G$-action on a compact Hausdorff space $X$, there is a largest factor $Y$ of $X$ with property $P$ \cite[Proposition 2.9.1]{GM}.
We prove a similar fact for the category of actions on compact metrizable spaces, the proof of which is implicit in the proof of \cite[Proposition 6.12]{Lindenstrauss}.

\begin{lemma} \label{L-existence of largest factor}
Let $\Gamma$ be a topological group. Let P be a property for continuous $\Gamma$-actions on compact metrizable spaces. Suppose that P is preserved by countable products, subsystems, and isomorphisms, and that the trivial action of $\Gamma$ on the one-point set $\bullet$ has property P.
Then any continuous $\Gamma$-action on a compact metrizable space $X$ has a largest factor $Y$ with property P,
i.e. for any factor $Z$ of $X$ with property P
there is a unique ($\Gamma$-equivariant continuous surjective) map $Y\rightarrow Z$ making the following diagram
\begin{eqnarray*}
\xymatrix{
X \ar[rd] \ar[r]
&Y \ar[d]\\
&Z}
\end{eqnarray*}
commute.
\end{lemma}
\begin{proof} For each factor $Z$ of $X$ with factor map $\pi_Z: X\rightarrow Z$, denote by $R_Z$ the closed subset $\{(x, y)\in X^2: \pi_Z(x)=\pi_Z(y)\}$
of $X^2$.  Denote by $R$ the set $\bigcap_Z R_Z$ for $Z$ ranging over factors of $X$ with property P.
Since $X^2$ is compact metrizable, it has a countable base. Thus every subset $W$ of $X^2$ with the topology inherited from $X^2$ is a Lindel\"{o}f space in the sense that every open cover of $W$ has a countable subcover \cite[page 49]{Kelley}. Taking $W=X^2\setminus R$ and considering the open cover of $X^2\setminus R$ consisting of  $X^2\setminus R_Z$ for all factors $Z$ of $X$ with property P,
we find factors $Z_1, Z_2, \dots$ of $X$ with property P such that $\bigcap_{n=1}^\infty R_{Z_n}=R$. Consider the map
$\pi: X\rightarrow \prod_{n=1}^\infty Z_n$ sending $x$ to $(\pi_{Z_n}(x))_{n=1}^\infty$. Then $Y:=\pi(X)$ is a closed $\Gamma$-invariant subset of $\prod_{n=1}^\infty Z_n$ and is a factor of $X$. Furthermore, $R_Y=R$. By the assumption on P and $Z_n$, we see that the $\Gamma$-action on $Y$ has property P.  Since $R_Z\supseteq R=R_Y$ for every factor $Z$ of $X$ with property P, clearly $Y$ is the largest factor of $X$ with property P.
\end{proof}

By Remark~\ref{R-mean top dim2} and Lemma~\ref{L-factor1} if $G$ acts continuously on a compact metrizable space $X$ with $\mdim_\Sigma(X)\ge 0$, then $\mdim_\Sigma(Y)\ge 0$ for every factor $Y$ of $X$.
From Lemma~\ref{L-existence of largest factor} and Propositions~\ref{P-subspace top} and \ref{P-product}, taking property P to be having sofic mean topological dimension at most $0$ and observing that $\mdim_\Sigma(\bullet)=0$ for the trivial action of $G$ on the one-point set $\bullet$, we obtain the following result, which was established by Lindenstrauss for countable amenable groups \cite[Proposition 6.12]{Lindenstrauss}.

\begin{proposition}
Let $G$ act continuously on a compact metrizable space $X$ with $\mdim_\Sigma(X)\ge 0$. Then $X$ has a largest factor $Y$ satisfying $\mdim_\Sigma(Y)=0$.
\end{proposition}

\section{Sofic Mean Topological Dimension for Amenable Groups} \label{S-top amenable}

In this section we show that the sofic mean topological dimension extends the mean topological dimension for actions of countably infinite amenable groups:

\begin{theorem} \label{T-top mean dim}
Let  a  countably infinite (discrete) amenable group $G$ act continuously on a compact metrizable space $X$.
Let $\Sigma$ be a sofic approximation sequence of $G$.
Then
$$\mdim_\Sigma(X)=\mdim(X).$$
\end{theorem}

Theorem~\ref{T-top mean dim} follows directly from Lemmas~\ref{L-top mean dim lower bound} and \ref{L-top mean dim upper bound} below.

We need the following Rokhlin lemma several times. Though for ergodic measure-preserving actions of $\Zb$ one needs only one Rokhlin tower, for actions of general countable amenable groups one needs several Rokhlin towers. Here one should think of $[d]$ as equipped with the uniform distribution. The assumption about $\sigma$ says that it is approximately a free action of $G$ on $[d]$ (preserving the uniform distribution), while the conclusion says that $\cC_1, \dots, \cC_\ell$ are the bases for Rokhlin towers and that for each $k=1, \dots, \ell$, $\{\sigma(s)\cC_k\}_{s\in F_k}$ is a Rokhlin tower.

\begin{lemma}\cite[Lemma 4.6]{KerLi10amenable} \label{L-Rokhlin}
Let $G$ be a countable amenable group.
Let $0\le \tau<1$, $0<\eta<1$, $\delta>0$, and $K$ be a nonempty finite subset of $G$.
Then there are an $\ell\in \Nb$, nonempty finite subsets $F_1, \dots, F_\ell$ of $G$ with $|KF_k \setminus F_k|<\delta |F_k|$ and
$|F_kK\setminus F_k|<\delta|F_k|$ for
all $k=1, \dots, \ell$, a finite set $F\subseteq G$ containing $e_G$,
and an $\eta'>0$ such that,
for every $d\in \Nb$, every map $\sigma: G\rightarrow \Sym(d)$ for which there is a set $\cB\subseteq [d]$ satisfying
$|\cB|\ge (1-\eta')d$ and
\[
\sigma_{st}(a)=\sigma_s\sigma_t(a), \sigma_s(a)\neq \sigma_{s'}(a), \sigma_{e_G}(a)=a
\]
for all $a\in \cB$ and $s, t, s'\in F$ with $s\neq s'$, and every set $\cW\subseteq [d]$ with $|\cW|\ge (1-\tau)d$,
there exist $\cC_1, \dots, \cC_\ell\subseteq \cW$ such that
\begin{enumerate}
\item for every $k=1, \dots, \ell$, the map $(s, c)\mapsto \sigma_s(c)$ from $F_k\times \cC_k$ to $\sigma(F_k)\cC_k$ is bijective,

\item the sets $\sigma(F_1)\cC_1, \dots, \sigma(F_\ell)\cC_\ell$ are pairwise disjoint and $|\bigcup_{k=1}^\ell \sigma(F_k)\cC_k|\ge (1-\tau-\eta)d$.
\end{enumerate}
\end{lemma}

\begin{remark} \label{R-Rokhlin}
Let $G$ be a finite group. Note that the only nonempty finite subset $F$ of $G$ satisfying $|GF\setminus F|<\frac{1}{|G|}|F|$ is $G$. From Lemma~\ref{L-Rokhlin} one deduces the following: Let $0\le \tau<1$ and $0<\eta<1$, then there is an $\eta'>0$ such that, for every $d\in \Nb$, every map $\sigma: G\rightarrow \Sym(d)$ for which there is $\cB\subseteq [d]$ satisfying
$|\cB|\ge (1-\eta')d$ and
\[
\sigma_{st}(a)=\sigma_s\sigma_t(a), \sigma_s(a)\neq \sigma_{s'}(a), \sigma_{e_G}(a)=a
\]
for all $a\in \cB$ and $s, t, s'\in G$ with $s\neq s'$, and every set $\cW\subseteq [d]$ with $|\cW|\ge (1-\tau)d$,
there exists $\cC\subseteq \cW$ such that the map $(s, c)\mapsto \sigma_s(c)$ from $G\times \cC$ to $\sigma(G)\cC$ is bijective and $|\sigma(G)\cC|\ge (1-\tau-\eta)d$.
\end{remark}

Combined with Lemma~\ref{L-Rokhlin}, the following lemma tells us how to construct elements in $\Map(\rho, F, \delta, \sigma)$.

\begin{lemma} \label{L-construct map}
Let $\alpha$ be a continuous action of a countable group $G$ on a compact metrizable space $X$. Let $\rho$ be a continuous pseudometric on $X$. Let $\delta, \delta'>0$ with
$\sqrt{\delta'}\diam(X, \rho)<\delta/2$.
Let $\ell\in \Nb$ and $F, F_1, \dots, F_\ell$ be nonempty finite subsets of $G$ with $|FF_k\setminus F_k|<\delta'|F_k|$ for all $k=1, \dots, \ell$. Let $\sigma$ be a map $G\rightarrow \Sym(d)$ for some $d\in \Nb$. Denote by $\cW$ the set of elements $a$ in $[d]$ satisfying
$\sigma_t\sigma_s(a)=\sigma_{ts}(a)$ for all $t\in F$ and $s\in \bigcup_{k=1}^\ell F_k$.  Suppose that there are $\cC_1, \dots, \cC_\ell\subseteq \cW$ satisfying the following:
\begin{enumerate}
\item for every $k=1, \dots, \ell$, the map $(s, c)\mapsto \sigma_s(c)$ from $F_k\times \cC_k$ to $\sigma(F_k)\cC_k$ is bijective,

\item the sets $\sigma(F_1)\cC_1, \dots, \sigma(F_\ell)\cC_\ell $ are pairwise disjoint and $|\bigcup_{k=1}^\ell\sigma(F_k)\cC_k|\ge (1-\delta')d$.
\end{enumerate}
For any $h=(h_k)_{k=1}^\ell\in \prod_{k=1}^\ell X^{\cC_k}$, if $\varphi: [d]\rightarrow X$ satisfies
$$ \varphi(sc)=s(h_k(c))$$
for all $k\in \{1, \dots, \ell\}, c\in \cC_k$, and $s\in F_k$, then $\varphi\in \Map(\rho, F, \delta, \sigma)$.
\end{lemma}
\begin{proof}
Note that if $t\in F$, $k\in \{1, \dots, \ell\}$, $s\in F_k$, $c\in \cC_k$,  and $ts\in F_k$, then $\sigma_t\sigma_s(c)=\sigma_{ts}(c)$, and hence  $\alpha_t\circ\varphi(sc)=\varphi\circ \sigma_t(sc)$.
For every $t\in F$, one has
\begin{align*}
 (\rho_2(\alpha_t\circ \varphi, \varphi\circ \sigma_t))^2&\le \frac{d-|\bigcup_{k=1}^\ell\sigma(F_k\cap t^{-1}F_k)\cC_k|}{d}(\diam(X, \rho))^2\\
 &= \frac{d-|\bigcup_{k=1}^\ell\sigma(F_k)\cC_k|+|\bigcup_{k=1}^\ell\sigma(F_k\setminus t^{-1}F_k)\cC_k|}{d}(\diam(X, \rho))^2\\
 &\le \frac{\delta'd+\delta'|\bigcup_{k=1}^\ell\sigma(F_k)\cC_k|}{d}(\diam(X, \rho))^2\\
 &\le 2\delta'(\diam(X, \rho))^2<\delta^2.
\end{align*}
Thus $\varphi \in\Map (\rho , F ,\delta , \sigma )$.
\end{proof}

We show first $\mdim_\Sigma (X)\ge \mdim(X)$ for infinite $G$. This amounts to show that $\Map(\rho, F, \delta, \sigma)$ is large enough. When $[d]=H$ for some approximately left invariant finite subset $H$ of $G$ and $\sigma_s\in \Sym(d)$ is essentially the left multiplication by $s$ for all $s\in F$, one can embed $X$ into $\Map(\rho, F, \delta, \sigma)$ sending $x$ to $(sx)_{s\in H}$. In general, $[d]$ may fail to be of the form $H$, but Lemma~\ref{L-Rokhlin} tells us $[d]$ is roughly the disjoint union of some such $H_j$ for $j\in J$. Then we do such embedding on each $H_j$ and hence embed $X^J$ into $\Map(\rho, F, \delta, \sigma)$.

\begin{lemma} \label{L-top mean dim lower bound}
Let a countably infinite  amenable group $G$ act continuously on a compact metrizable space $X$.
Then for any finite open cover $\cU$ of $X$ we have $\cD_\Sigma(\cU)\ge \mdim(\cU)$. In particular,
$\mdim_\Sigma (X)\ge \mdim(X)$.
\end{lemma}
\begin{proof}It suffices to show that $\cD_\Sigma(\cU)\ge \mdim(\cU)-2\theta$ for every $\theta>0$.

Fix a compatible metric $\rho$ on $X$. Let $F$ be a nonempty finite subset of $G$ and $\delta > 0$. Let $\sigma$ be a map from $G$ to $\Sym (d)$
for some $d\in\Nb$.
Now it suffices to show that if $\sigma$ is a good enough sofic approximation then
\begin{align*}
\frac{\cD(\cU, \rho, F, \delta, \sigma)}{d} \ge \mdim(\cU) - 2\theta .
\end{align*}

Take a finite subset $K$ of $G$ containing $F$ and $\varepsilon>0$ such that for any nonempty finite subset $F'$ of $G$ with $|KF'\setminus F'|<\varepsilon |F'|$ one has
$$ \frac{\cD(\cU^{F'})}{|F'|}\ge \mdim(\cU)-\theta.$$

Take $0<\delta'<1$ to be small such that
$(\mdim(\cU)-\theta)(1-\delta')\ge \mdim(\cU) - 2\theta$ and $\sqrt{\delta'}\diam(X, \rho)<\delta/2$.
By Lemma~\ref{L-Rokhlin} there are an $\ell\in \Nb$
and nonempty finite subsets $F_1, \dots, F_\ell$ of $G$
satisfying $|KF_k\setminus F_k|<\min(\varepsilon, \delta') |F_k|$ for all $k=1, \dots, \ell$
such that for every map $\sigma : G\to\Sym(d)$ for some $d\in \Nb$ which is a good
enough sofic approximation for $G$ and every $\cW\subseteq [d]$ with $|\cW|\ge (1-\delta'/2)d$ there exist  $\cC_1, \dots, \cC_\ell\subseteq \cW$ satisfying the following:
\begin{enumerate}
\item for every $k=1, \dots, \ell$, the map $(s, c)\mapsto \sigma_s(c)$ from $F_k\times \cC_k$ to $\sigma(F_k)\cC_k$ is bijective,

\item the sets $\sigma(F_1)\cC_1, \dots, \sigma(F_\ell)\cC_\ell $ are pairwise disjoint and $|\bigcup_{k=1}^\ell\sigma(F_k)\cC_k|\ge (1-\delta')d$.
\end{enumerate}

Let $\sigma: G\rightarrow \Sym(d)$ for some $d\in \Nb$ be a good enough sofic approximation for $G$ such that $|\cW|\ge (1-\delta'/2)d$ for
$$ \cW:=\{a\in [d]: \sigma_t\sigma_s(a)=\sigma_{ts}(a) \mbox{ for all } t\in F, s\in \bigcup_{k=1}^\ell F_k\}.$$
Then we have $\cC_1, \dots, \cC_\ell$ as above.

Since $G$ is infinite, there exist maps $\psi_k: \cC_k\rightarrow G$ for  $k=1, \dots, \ell$ such that
the map $\Psi$ from $\bigsqcup_{k=1}^\ell F_k\times \cC_k$ to $G$ sending $(s, c)\in F_k\times \cC_k$ to $s\psi_k(c)$ is injective.
Denote by $\tilde{F}$ the range of $\Psi$. Note that $|K\tilde{F}\setminus \tilde{F}|<\varepsilon |\tilde{F}|$ because for every $k$, $|KF_k\setminus F_k|<\varepsilon |F_k|$. Thus
$$  \frac{\cD(\cU^{\tilde{F}})}{|\tilde{F}|}\ge \mdim(\cU)-\theta.$$

Pick $x_0\in X$. For each $x\in X$ define a map $\varphi_x: [d]\rightarrow X$ by $\varphi_x(a)=x_0$ for all $a\in [d]\setminus \bigcup_{k=1}^\ell\sigma(F_k)\cC_k$, and
$$ \varphi_x(sc)=s\psi_k(c)x$$
for all $k\in \{1, \dots, \ell\}$, $c\in \cC_k$, and $s\in F_k$.
By Lemma~\ref{L-construct map} one has $\varphi_x \in\Map (\rho , F ,\delta , \sigma )$.

Note that the map $\Phi$ from $X$ to $\Map (\rho , F ,\delta , \sigma )$ sending $x$ to $\varphi_x$ is continuous, and $\Phi^{-1}(\cU^d|_{\Map (\rho , F ,\delta, \sigma )})=\cU^{\tilde{F}}$. Thus $\cD(\cU, \rho, F, \delta, \sigma)\ge \cD(\cU^{\tilde{F}})$. Therefore
\begin{align*}
\frac{\cD(\cU, \rho, F, \delta, \sigma)}{d} \ge \frac{\cD(\cU^{\tilde{F}})}{|\tilde{F}|}\cdot \frac{|\tilde{F}|}{d}\ge (\mdim(\cU)-\theta)(1-\delta')\ge \mdim(\cU) - 2\theta,
\end{align*}
as desired.
\end{proof}

Let $\cU$ be a finite open cover of a compact metrizable space $X$. A continuous map $f$ from $X$ into another compact metrizable space $Y$ is said to be
{\it $\cU$-compatible} if for each $y\in Y$, the set $f^{-1}(y)$ is contained in some $U\in \cU$ \cite[Definition 2.2 and Proposition 2.3]{LW}. We need the following fact:

\begin{lemma}\cite[Proposition 2.4]{LW} \label{L-compatible map}
Let $\cU$ be a finite open cover of a compact metrizable space $X$, and $k\ge 0$. Then $\cD(\cU)\le k$ if and only if there is a $\cU$-compatible continuous map
$f: X\rightarrow Y$ for some  compact metrizable space $Y$ with dimension $k$.
\end{lemma}

Next we show $\mdim_\Sigma(X)\le \mdim(X)$.
We first use Lemma~\ref{L-Rokhlin} to decompose $[d]$ into the disjoint union of some approximately left invariant finite subsets $\{H_j\}_{j\in J}$ of $G$ and a small portion $[d]\setminus \bigcup_jH_j$. Take a continuous $\cU^{H_j}$-compatible map $X\rightarrow Y_j$ with $\dim(Y_j)\le \cD(\cU^{H_j})$ for each $j\in J$. Anticipating each element of $\Map(\rho, F, \delta, \sigma)$ being essentially of the form $(sx_j)_{s\in H_j}$ with some $x_j\in X$ on $H_j$ for each $j\in J$, we map
$\Map(\rho, F, \delta, \sigma)$ to $\prod_jY_j$.  To take care of the coordinates on $[d]\setminus \bigcup_jH_j$, we also take a continuous $\cU$-compatible map $X\rightarrow Z$ with $\dim(Z)\le \cD(\cU)$,  and map $\Map(\rho, F, \delta, \sigma)$ to $\prod_{a\in [d]\setminus \bigcup_jH_j}Z$. These two maps combined together are $\cU^d|_{\Map(\rho, F, \delta, \sigma)}$-compatible on nice elements which are almost of the form $(sx_j)_{s\in H_j}$ with some $x_j\in X$ on $H_j$ for each $j\in J$.
The points of $\Map(\rho, F, \delta, \sigma)$ not so nice can be bad only at a small portion of $[d]$. Then we use some auxiliary map $h$ to shrink $X$ to one point at all the good places, making the bad parts to live in a small-dimensional space (relative to $d$).

\begin{lemma} \label{L-top mean dim upper bound}
Let a countable amenable group $G$ act continuously on a compact metrizable space $X$.
Then
$\mdim_\Sigma(X)\le \mdim(X)$.
\end{lemma}
\begin{proof} Fix a compatible metric $\rho$ on $X$.
Let $\cU$ be a finite open cover of $X$. It suffices to show that $\cD_\Sigma(\cU)\le \mdim(X)$.

Take a finite open cover $\cV$ of $X$ such that for every $V\in \cV$, one has $\overline{V}\subseteq U$ for some $U\in \cU$. Then it suffices to show
$\cD_\Sigma(\cU)\le \mdim(\cV)+3\theta$ for every $\theta>0$. We can find $\eta>0$ such that for every $V\in \cV$, one has $B(V, \eta)=\{x\in X: \rho(x, V)<\eta\}\subseteq U$ for some $U\in \cU$.

Take a nonempty finite subset $K$ of $G$ and $\varepsilon>0$ such that for any nonempty finite subset $F'$ of $G$ with $|KF'\setminus F'|<\varepsilon |F'|$ one has
$$ \frac{\cD(\cV^{F'})}{|F'|}\le \mdim(\cV)+\theta.$$

Take $\tau>0$ with $\tau \cD(\cU)\le \theta$.
By Lemma~\ref{L-Rokhlin}  there are an $\ell\in \Nb$
and nonempty finite subsets $F_1, \dots, F_\ell$ of $G$
satisfying $|KF_k\setminus F_k|<\varepsilon |F_k|$ for all $k=1, \dots, \ell$
such that for every map $\sigma : G\to\Sym(d)$ for some $d\in \Nb$ which is a good
enough sofic approximation for $G$  and every $\cW\subseteq [d]$ with $|\cW|\ge (1-\tau/2)d$ there exist  $\cC_1, \dots, \cC_\ell\subseteq \cW$ satisfying the following:
\begin{enumerate}
\item for every $k=1, \dots, \ell$, the map $(s, c)\mapsto \sigma_s(c)$ from $F_k\times \cC_k$ to $\sigma(F_k)\cC_k$ is bijective,

\item the sets $\sigma(F_1)\cC_1, \dots, \sigma(F_\ell)\cC_\ell$ are pairwise disjoint and $|\bigcup_{k=1}^\ell \sigma(F_k)\cC_k|\ge (1-\tau)d$.
\end{enumerate}

Set $F=\bigcup_{k=1}^\ell F_k^{-1}$. Take $\kappa>0$ such that for any $x, y\in X$ with $\rho(x, y)\le \kappa$ one has $\rho(s^{-1}x, s^{-1}y)<\eta$ for all $s\in F$.
Take $\delta>0$ with $\delta^{1/2}<\kappa$ and $\delta |\cU| |F|\le \theta$.

Let $\sigma$ be a map from $G$ to $\Sym (d)$
for some $d\in\Nb$.
Now it suffices to show that if $\sigma$ is a good enough sofic approximation then
\begin{align*}
\frac{\cD(\cU, \rho, F, \delta, \sigma)}{d} \le \mdim(\cV) + 3\theta .
\end{align*}

Denote by $\cW$ the subset of $[d]$ consisting of $a$ satisfying $\sigma_s\sigma_{s^{-1}}(a)=\sigma_{e_G}(a)=a$ for all $s\in F$.
Assuming that $\sigma$ is a good enough sofic approximation,  we have $|\cW|\ge (1-\tau/2)d$ and can find $\cC_1, \dots, \cC_\ell$ as above.
Set $\cZ=[d]\setminus \bigcup_{k=1}^\ell \sigma(F_k)\cC_k$. Then $|\cZ|\le \tau d$.

For every $\varphi\in\Map (\rho ,F, \delta, \sigma )$,
by Lemma~\ref{L-almost}
the set $\Lambda_{\varphi}$
of all $a\in [d]$ satisfying
\[
\rho(\varphi (sa), s\varphi(a)) \le \delta^{1/2}
\]
for all $s\in F$ has cardinality at least $(1-|F|\delta) d$.

Take a partition of unity $\{\zeta_U\}_{U\in \cU}$ for $X$ subordinate  to $\cU$. That is, each $\zeta_U$ is a continuous function $X\rightarrow [0, 1]$ with support contained in $U$, and
$$ \sum_{U\in \cU}\zeta_U=1.$$
Define a continuous map $\overrightarrow{\xi}: X \rightarrow [0, 1]^{\cU}$ by $\overrightarrow{\xi}(x)_U=\zeta_U(x)$ for $x\in X$ and $U
\in \cU$. Consider the continuous map $\overrightarrow{h}: \Map(\rho, F, \delta, \sigma) \rightarrow ([0, 1]^{\cU})^{[d]}$ defined by
$$\overrightarrow{h}(\varphi)_a=\overrightarrow{\xi}(\varphi(a))\max(\max_{s\in F}\rho(s \varphi(a), \varphi(sa))-\kappa, 0)$$
for $\varphi \in \Map(\rho, F, \delta, \sigma)$ and $a\in [d]$.
Denote by $\nu$ the point of $[0, 1]^{\cU}$ having all coordinates $0$.
Set $X_0$ to be the subset of  $([0, 1]^{\cU})^{[d]}$ consisting of elements whose coordinates are equal to $\nu$ at at least $(1-|F|\delta) d$  elements of $[d]$.
For each $\varphi\in \Map(\rho, F, \delta, \sigma)$, note that  $\overrightarrow{h}(\varphi)_a=\nu$ for all $a\in \Lambda_\varphi$ by our choice of $\delta$ and hence $\overrightarrow{h}(\varphi)\in X_0$. Thus we may think of $\overrightarrow{h}$ as a map from $\Map(\rho, F, \delta, \sigma)$ into $X_0$. Since the union of finitely many closed subsets of dimension at most $m$ has dimension at most $m$ \cite[page 30 and Theorem V.8]{HW} and $\delta>0$ was chosen such that $|\cU| |F| \delta\le \theta$, we get $\dim(X_0)\le |\cU| |F|\delta d\le \theta d$.

For each $1\le k\le \ell$, by Lemma~\ref{L-compatible map} we can find a compact metrizable space $Y_k$ with $\dim(Y_k)\le \cD(\cV^{F_k})$ and a $\cV^{F_k}$-compatible continuous map $f_k:X\rightarrow Y_k$.

By Lemma~\ref{L-compatible map} we can find a compact metrizable space $Z$ with $\dim(Z)\le \cD(\cU)$ and a $\cU$-compatible continuous map $g:X\rightarrow Z$.

Now define a continuous map $\Psi: \Map(\rho, F, \delta, \sigma)\rightarrow X_0\times (\prod_{k=1}^\ell\prod_{c\in \cC_k}Y_k)\times (\prod_{a\in \cZ}Z)$
as follows. For $\varphi\in \Map(\rho, F, \delta, \sigma)$, the coordinate of $\Psi(\varphi)$ in $X_0$ is $\overrightarrow{h}(\varphi)$, in $Y_k$ for $1\le k\le \ell$ and $c\in \cC_k$ is $f_k(\varphi(c))$, in $Z$ for $a\in \cZ$ is $g(\varphi(a))$. We claim that $\Psi$ is $\cU^d|_{\Map(\rho, F, \delta, \sigma)}$-compatible.
Let $w\in X_0\times (\prod_{k=1}^\ell\prod_{c\in \cC_k}Y_k)\times (\prod_{a\in \cZ}Z)$. We need to show that for each $a\in [d]$ there is some
$U\in \cU$ depending only on $w$ and $a$ such that $\varphi(a)\in U$ for every  $\varphi \in \Psi^{-1}(w)$.
We write the coordinates of $w$ in  $X_0$, $\prod_{k=1}^\ell\prod_{c\in \cC_k}Y_k$, and $\prod_{a\in \cZ}Z$ as $w^1$, $w^2$, and $w^3$ respectively.

For each $a\in \cZ$, since $g$ is $\cU$-compatible, one has $g^{-1}(w^3_a)\subseteq U_{w^3_a}$ for some $U_{w^3_a}\in \cU$. Then $\varphi(a)\in U_{w^3_a}$ for every
$\varphi \in \Psi^{-1}(w)$ and
$a\in \cZ$.

For every $1\le k\le \ell$ and $c\in \cC_k$, since $f_k$ is $\cV^{F_k}$-compatible, one has $f_k^{-1}(w^2_{k, c})\subseteq \bigcap_{s^{-1}\in F_k}sV_{k, c, s}$ for some $V_{k, c, s}\in \cV$ for every $s^{-1}\in F_k$. By the choice of $\eta$,
$B(V_{k, c, s}, \eta)$ is contained in some $U_{k, c, s}\in \cU$.
For every $a\in [d]\setminus \cZ$, we distinguish the two cases $w^1_a\neq \nu$ and $w^1_a=\nu$. If $w^1_a\neq \nu$, then $(w^1_a)_U\neq 0$ for some $U\in \cU$, and then for $\varphi \in \Psi^{-1}(w)$ one has $\zeta_U(\varphi(a))>0$ and hence  $\varphi(a)\in U$.
Suppose that $w^1_a=\nu$.  Say, $a=\sigma(s^{-1})c$ for some $1\le k\le \ell$, $s^{-1}\in F_k$ and $c\in \cC_k$. Let $\varphi \in \Psi^{-1}(w)$.
Since $c\in \cC_k\subseteq \cW$ and $s\in F$, one has $sa=\sigma_{s}\sigma_{s^{-1}}(c)=c$.
As $\{\zeta_U\}_{U\in \cU}$ is a partition of unity of $X$, $\overrightarrow{\xi}(\varphi(a))\neq \nu$. But $\overrightarrow{h}(\varphi)_a=w^1_a=\nu$.
Thus $\max_{s'\in F}\rho(s' \varphi(a), \varphi(s'a))\le \kappa$. In particular,
one has
$\rho(s\varphi(a), \varphi(c))=\rho(s\varphi(a), \varphi(sa))\le \kappa$. From our choice of $\kappa$, one gets $\rho(\varphi(a), s^{-1}\varphi(c))<\eta$.
Since $f_k(\varphi(c))=w^2_{k, c}$, we have $\varphi(c)\in sV_{k, c, s}$ and hence $s^{-1}\varphi(c)\in V_{k, c, s}$. Therefore $\varphi(a)\in U_{k, c, s}$.
This proves the claim.

From Lemma~\ref{L-compatible map}  we get
$$ \cD(\cU, \rho, F, \delta, \sigma) \le \dim\big(X_0\times (\prod_{k=1}^\ell\prod_{c\in \cC_k}Y_k)\times (\prod_{a\in \cZ}Z)\big).$$
Since the dimension of the product of two compact metrizable spaces is at most the sum of the dimensions of the factors \cite[page 33 and Theorem V.8]{HW},  we have
\begin{align*}
 \dim\big(X_0\times (\prod_{k=1}^\ell\prod_{c\in \cC_k}Y_k)\times (\prod_{a\in \cZ}Z)\big) &\le \dim(X_0)+\sum_{k=1}^\ell |\cC_k|\dim(Y_k)+|\cZ|\dim(Z)\\
 &\le \theta d+\sum_{k=1}^\ell |\cC_k|\cD(\cV^{F_k}) +|\cZ|\cD(\cU)\\
 &\le \theta d+\sum_{k=1}^\ell |\cC_k| |F_k|(\mdim(\cV)+\theta) +\tau d\cD(\cU)\\
 &\le \theta d+d(\mdim(\cV)+\theta) +\theta d\\
 &= d(\mdim(\cV)+3\theta).
\end{align*}
Therefore $\cD(\cU, \rho, F, \delta, \sigma) \le d(\mdim(\cV)+3\theta)$ as desired.
\end{proof}

\begin{remark} \label{R-sofic top not equal to top for finite}
Theorem~\ref{T-top mean dim} fails when $G$ is finite. Indeed, when a finite group $G$ acts continuously on a compact metrizable space $X$,
one has $\mdim(X)=\frac{1}{|G|}\dim(X)$. There are compact metrizable finite-dimensional spaces $X$ satisfying $\dim(X^2)<2\dim(X)$ (see \cite{Boltyanskii}). For such $X$,
Lemma~\ref{L-top mean dim upper bound finite} below implies that $\mdim_\Sigma(X)<\mdim(X)$.
\end{remark}

If $X$ is a compact metrizable space with finite dimension, then for any $n, m\in \Nb$ one has
$\dim(X^n\times X^m)\le \dim(X^n)+\dim(X^m)$ \cite[page 33 and Theorem V.8]{HW} and hence $\frac{\dim(X^n)}{n}\to \inf_{m\in \Nb}\frac{\dim(X^m)}{m}$ as $n\to \infty$.

\begin{lemma} \label{L-top mean dim upper bound finite}
Let a finite group $G$ act continuously on a compact metrizable finite-dimensional space $X$.
Then
$\mdim_\Sigma(X)\le \frac{1}{|G|}\inf_{m\in \Nb}\frac{\dim(X^m)}{m}$.
\end{lemma}
\begin{proof} The proof is similar to that of Lemma~\ref{L-top mean dim upper bound}. Fix a compatible metric $\rho$ on $X$.
Set $\lambda=\frac{1}{|G|}\inf_{m\in \Nb}\frac{\dim(X^m)}{m}$.
Let $\cU$ be a finite open cover of $X$ and $\theta>0$. It suffices to show that $\cD_\Sigma(\cU)\le \lambda+3\theta$.

We can find $\eta>0$ such that for every $y\in X$, one has $\{x\in X: \rho(x, y)<\eta\}\subseteq U$ for some $U\in \cU$.

Take $M>0$ such that $\frac{\dim(X^m)}{m}\le (\lambda+\theta) |G|$ for all $m\ge M$.

Take $\tau>0$ with $\tau \dim(X)\le \theta$.
By Remark~\ref{R-Rokhlin}
for every map $\sigma : G\to\Sym(d)$ for some $d\in \Nb$ which is a good
enough sofic approximation for $G$  and every $\cW\subseteq [d]$ with $|\cW|\ge (1-\tau/2)d$ there exists  $\cC\subseteq \cW$
such that the map $(s, c)\mapsto \sigma_s(c)$ from $G\times \cC$ to $\sigma(G)\cC$ is bijective and
$|\sigma(G)\cC|\ge (1-\tau)d$.

Take $\kappa>0$ such that for any $x, y\in X$ with $\rho(x, y)\le \kappa$ one has $\rho(sx, sy)<\eta$ for all $s\in G$.
Take $\delta>0$ with $\delta^{1/2}<\kappa$ and $\delta |\cU| |G|\le \theta$.

Let $\sigma$ be a map from $G$ to $\Sym (d)$
for some $d\in\Nb$.
Now it suffices to show that if $\sigma$ is a good enough sofic approximation and $d$ is sufficiently large then
\begin{align*}
\frac{\cD(\cU, \rho, G, \delta, \sigma)}{d} \le \lambda + 3\theta .
\end{align*}

Denote by $\cW$ the subset of $[d]$ consisting of $a$ satisfying $\sigma_s\sigma_{s^{-1}}(a)=\sigma_{e_G}(a)=a$ for all $s\in G$.
Assuming that $\sigma$ is a good enough sofic approximation,  we have $|\cW|\ge (1-\tau/2)d$ and can find $\cC$ as above.
Set $\cZ=[d]\setminus \sigma(G)\cC$. Then $|\cZ|\le \tau d$.

For every $\varphi\in\Map (\rho, G, \delta, \sigma )$,
by Lemma~\ref{L-almost}
the set $\Lambda_{\varphi}$
of all $a\in [d]$ satisfying
\[
\rho(\varphi (sa), s\varphi(a)) \le \delta^{1/2}
\]
for all $s\in G$ has cardinality at least $(1-|G|\delta) d$.

Take a partition of unity $\{\zeta_U\}_{U\in \cU}$ for $X$ subordinate  to $\cU$. That is, each $\zeta_U$ is a continuous function $X\rightarrow [0, 1]$ with support contained in $U$, and
$$ \sum_{U\in \cU}\zeta_U=1.$$
Define a continuous map $\overrightarrow{\xi}: X \rightarrow [0, 1]^{\cU}$ by $\overrightarrow{\xi}(x)_U=\zeta_U(x)$ for $x\in X$ and $U
\in \cU$. Consider the continuous map $\overrightarrow{h}: \Map(\rho, G, \delta, \sigma) \rightarrow ([0, 1]^{\cU})^{[d]}$ defined by
$$\overrightarrow{h}(\varphi)_a=\overrightarrow{\xi}(\varphi(a))\max(\max_{s\in G}\rho(s \varphi(a), \varphi(sa))-\kappa, 0)$$
for $\varphi \in \Map(\rho, G, \delta, \sigma)$ and $a\in [d]$.
Denote by $\nu$ the point of $[0, 1]^{\cU}$ having all coordinates $0$.
Set $X_0$ to be the subset of  $([0, 1]^{\cU})^{[d]}$ consisting of elements whose coordinates are equal to $\nu$ at at least $(1-|G|\delta) d$ elements of $[d]$.
For each $\varphi\in \Map(\rho, G, \delta, \sigma)$, note that  $\overrightarrow{h}(\varphi)_a=\nu$ for all $a\in \Lambda_\varphi$ by our choice of $\delta$ and hence $\overrightarrow{h}(\varphi)\in X_0$. Thus we may think of $h$ as a map from $\Map(\rho, G, \delta, \sigma)$ into $X_0$. Since the union of finitely many closed subsets of dimension at most $m$ has dimension at most $m$ \cite[page 30]{HW} and $\delta>0$ was chosen such that $|\cU| |G|\delta\le \theta$, we get $\dim(X_0)\le |\cU| |G|\delta d\le \theta d$.

Now define a continuous map $\Psi: \Map(\rho, G, \delta, \sigma)\rightarrow X_0\times (\prod_{c\in \cC}X)\times (\prod_{a\in \cZ}X)$
as follows. For $\varphi\in \Map(\rho, G, \delta, \sigma)$, the coordinate of $\Psi(\varphi)$ in $X_0$ is $\overrightarrow{h}(\varphi)$, in $X$ for $c\in \cC$ is $\varphi(c)$, in $X$ for $a\in \cZ$ is $\varphi(a)$. We claim that $\Psi$ is $\cU^d|_{\Map(\rho, G, \delta, \sigma)}$-compatible.
Let $w\in X_0\times (\prod_{c\in \cC}X)\times (\prod_{a\in \cZ}X)$. We need to show that for each $a\in [d]$ there is some
$U\in \cU$ depending only on $w$ and $a$ such that $\varphi(a)\in U$ for every  $\varphi \in \Psi^{-1}(w)$.
We write the coordinates of $w$ in  $X_0$, $\prod_{c\in \cC}X$, and $\prod_{a\in \cZ}X$ as $w^1$, $w^2$, and $w^3$ respectively.

For each $a\in \cZ$, one has $w^3_a\in U_{w^3_a}$ for some $U_{w^3_a}\in \cU$. Then $\varphi(a)\in U_{w^3_a}$ for every
$\varphi \in \Psi^{-1}(w)$ and
$a\in \cZ$.

For every $a\in [d]\setminus \cZ$, we distinguish the two cases $w^1_a\neq \nu$ and $w^1_a=\nu$. If $w^1_a\neq \nu$, then $(w^1_a)_U\neq 0$ for some $U\in \cU$, and then for $\varphi \in \Psi^{-1}(w)$ one has $\zeta_U(\varphi(a))>0$ and hence  $\varphi(a)\in U$.
Suppose that $w^1_a=\nu$.  Say, $a=\sigma(s^{-1})c$ for some $s^{-1}\in G$ and $c\in \cC$. Then $\{x\in X: \rho(x, s^{-1}\varphi(c))<\eta\}\subseteq U$ for some $U\in \cU$. Let $\varphi \in \Psi^{-1}(w)$.
Since $c\in \cC\subseteq \cW$, one has $sa=\sigma_{s}\sigma_{s^{-1}}(c)=c$.
As $\{\zeta_U\}_{U\in \cU}$ is a partition of unity of $X$, $\overrightarrow{\xi}(\varphi(a))\neq \nu$. But $h(\varphi)_a=w^1_a=\nu$.
Thus $\max_{s'\in G}\rho(s' \varphi(a), \varphi(s'a))\le \kappa$. In particular,
one has
$\rho(s\varphi(a), \varphi(c))=\rho(s\varphi(a), \varphi(sa))\le \kappa$. From our choice of $\kappa$, one gets $\rho(\varphi(a), s^{-1}\varphi(c))<\eta$.
Thus $\varphi(a)\in U$.
This proves the claim.

From Lemma~\ref{L-compatible map} we get
$$ \cD(\cU, \rho, G, \delta, \sigma) \le \dim\big(X_0\times (\prod_{c\in \cC}X)\times (\prod_{a\in \cZ}X)\big).$$
Taking $d$ to be sufficiently large, we have $|\cC|\ge M$ and hence $\dim(X^{|\cC|})\le |\cC||G|(\lambda +\theta)$.
Since the dimension of the product of two compact metrizable spaces is at most the sum of the dimensions of the factors \cite[page 33 and Theorem V.8]{HW},  we have
\begin{align*}
 \dim\big(X_0\times (\prod_{c\in \cC}X)\times (\prod_{a\in \cZ}X)\big) &\le \dim(X_0)+\dim(X^{|\cC|})+|\cZ|\dim(X)\\
 &\le \theta d+|\cC||G|(\lambda +\theta) +\tau d \dim(X)\\
 &\le \theta d+d(\lambda+\theta) +\theta d\\
 &= d(\lambda+3\theta).
\end{align*}
Therefore $\cD(\cU, \rho, G, \delta, \sigma) \le d(\lambda+3\theta)$ as desired.
\end{proof}

\section{Sofic Metric Mean Dimension} \label{S-sofic metric mean dim}

In this section we define the sofic metric mean dimension and establish some basic properties for it.

We start with recalling the definitions of the lower box dimension for a compact metric space and the metric mean dimension for actions of countable amenable groups.

For a pseudometric space $(Y, \rho)$ and $\varepsilon>0$ we say a subset $Z$ of $Y$ is {\it $(\rho, \varepsilon)$-separated} if $\rho(y, z)\ge \varepsilon$ for all distinct $y, z\in Z$. Denote by $N_\varepsilon(Y, \rho)$ the maximal cardinality of $(\rho, \varepsilon)$-separated subsets of $Y$.

The {\it lower box dimension} of a compact metric space $(Y, \rho)$ is defined as
$$ \underline{\dim}_B(Y, \rho):=\varliminf_{\varepsilon\to 0}\frac{\log N_{\varepsilon}(Y, \rho)}{|\log \varepsilon|}.$$

Let a countable (discrete) amenable group $G$ act continuously on a compact metrizable space $X$. Let $\rho$ be a continuous pesudometric on $X$.
For a finite open cover $\cU$ of $X$, we define the mesh of $\cU$ under $\rho$ by
$$ \mesh(\cU, \rho)=\max_{U\in \cU}\diam(U, \rho).$$
For a nonempty finite subset $F$ of $G$, we define a pseudometric $\rho_F$ on $X$ by
$$ \rho_F(x, y)=\max_{s\in F}\rho(sx, sy)$$
for $x, y\in X$. The function $F\mapsto \log \min_{\mesh(\cU, \rho_F)<\varepsilon} |\cU|$ defined on the set of nonempty finite subsets of $G$ satisfies
the conditions of the Ornstein-Weiss lemma \cite{OW} \cite[Theorem 6.1]{LW}, thus $\min_{\mesh(\cU, \rho_F)<\varepsilon} \frac{\log |\cU|}{|F|}$ converges to some real number, denoted by $S(X, \varepsilon, \rho)$, as $F$ becomes more and more left invariant. The {\it metric mean  dimension of $X$ with respect to $\rho$} \cite[page 13]{LW} is defined as
$$\mdim_\rM(X, \rho)=\varliminf_{\varepsilon\to 0}\frac{S(X, \varepsilon, \rho)}{|\log \varepsilon|}.$$
For any nonempty finite subset $F$ of $G$ and $\varepsilon>0$, it is easy to check that
\begin{align} \label{E-separated vs spanning}
 \min_{\mesh(\cU, \rho_F)<\varepsilon}|\cU|\ge N_\varepsilon(X, \rho_F)\ge \min_{\mesh(\cU, \rho_F)<2\varepsilon}|\cU|.
\end{align}
As discussed on page 14 of \cite{LW}, using \eqref{E-separated vs spanning} one can also write $\mdim_\rM(X, \rho)$ in a way similar to $\underline{\dim}_B(Y, \rho)$:
$$ \mdim_\rM(X, \rho)=\varliminf_{\varepsilon\to 0}\frac{1}{|\log \varepsilon|}\varlimsup_{n\to \infty}\frac{\log N_\varepsilon(X, \rho_{F_n})}{|F_n|}$$
for any left F{\o}lner sequence $\{F_n\}_{n\in \Nb}$ of $G$ (see Definition~\ref{D-amenable}).
Also, define
$$ \mdim_\rM(X)=\inf_\rho\mdim_\rM(X, \rho)$$
for $\rho$ ranging over compatible metrics on $X$.

In the rest of this section we fix a countable sofic group $G$ and
a sofic approximation sequence $\Sigma = \{ \sigma_i : G \to \Sym (d_i ) \}_{i=1}^\infty$ for $G$.
We also fix  a continuous action $\alpha$ of $G$ on a compact metrizable space $X$.

As we discussed in Section~\ref{S-sofic mean topological dim}, when defining sofic invariants, we replace $X$ by $\Map(\rho, F, \delta, \sigma)$.
We also replace $(\rho_F, \varepsilon)$-separated subsets of $X$ by $(\rho_\infty, \varepsilon)$-separated subsets of $\Map(\rho, F, \delta, \sigma)$.
Though we are mainly interested in the sofic metric mean dimension with respect to a metric, sometimes it can be calculated using certain continuous pseudometric as Lemma~\ref{L-allow pseudometric} and  Section~\ref{S-shifts} indicate. Thus we start various definitions for continuous pseudometrics.

\begin{definition} \label{D-sofic metric mean dim}
Let $F$ be a nonempty finite subset of $G$ and
$\delta > 0$. For $\varepsilon > 0$ and $\rho$ a continuous pseudometric on $X$ we define
\begin{align*}
h_{\Sigma ,\infty}^\varepsilon (\rho ,F, \delta ) &=
\varlimsup_{i\to\infty} \frac{1}{d_i} \log N_\varepsilon (\Map (\rho ,F,\delta ,\sigma_i ),\rho_\infty ) ,\\
h_{\Sigma ,\infty}^\varepsilon (\rho ,F) &= \inf_{\delta > 0} h_{\Sigma ,\infty}^\varepsilon (\rho ,F,\delta ) ,\\
h_{\Sigma ,\infty}^\varepsilon (\rho ) &= \inf_{F} h_{\Sigma ,\infty}^\varepsilon (\rho ,F) ,
\end{align*}
where $F$ in the third line ranges over the nonempty finite subsets of $G$.
If $\Map (\rho ,F,\delta ,\sigma_i )$ is empty for all sufficiently large $i$, we set
$h_{\Sigma ,\infty}^\varepsilon (\rho ,F, \delta ) = -\infty$. We define the {\it sofic metric mean dimension of $\alpha$ with respect to
$\rho$}  as
$$\mdim_{\Sigma, \rM} (X, \rho ) = \varliminf_{\varepsilon\to 0} \frac{1}{|\log \varepsilon|}h_{\Sigma ,\infty}^\varepsilon (\rho ).$$
We also define
$$ \mdim_{\Sigma, \rM}(X) =\inf_{\rho} \mdim_{\Sigma, \rM} (X, \rho ),$$
for $\rho$ ranging over compatible metrics on $X$. Note that $\mdim_{\Sigma, \rM}(\cdot, \cdot)$ is an invariant of metric dynamical systems, while $\mdim_{\Sigma, \rM}(\cdot)$ is an invariant of topological dynamical system.
\end{definition}

\begin{remark} \label{R-mean metric dim1}
Note that $h_{\Sigma ,\infty}^\varepsilon (\rho ,F, \delta )$ decreases when $\delta$ decreases and $F$ increases. Thus in the definitions of $h_{\Sigma ,\infty}^\varepsilon (\rho ,F)$ and $h_{\Sigma ,\infty}^\varepsilon (\rho )$ one can also replace $\inf_{\delta>0}$ and $\inf_F$ by $\lim_{\delta\to 0}$ and $\lim_{F\to \infty}$ respectively, where $F_1\le F_2$ means  $F_1\subseteq F_2$.
\end{remark}

The {\it sofic topological entropy} $h_\Sigma(X)$ of $\alpha$ was defined in \cite[Definition 4.6]{KL11}. It was shown in Proposition 2.4 of
\cite{KerLi10amenable} that
$$ h_\Sigma(X)=\lim_{\varepsilon \to 0}h^\varepsilon_{\Sigma, \infty}(\rho)$$
for every compatible metric $\rho$ on $X$. Thus we have

\begin{proposition} \label{P-entropy vs dim}
If $h_\Sigma(X)<+\infty$, then $\mdim_{\Sigma, \rM}(X, \rho)\le 0$ for every compatible metric $\rho$ on $X$.
\end{proposition}

The amenable group case of Proposition~\ref{P-entropy vs dim} was observed by Lindenstrauss and Weiss \cite[page 14]{LW}.

We say that a continuous pseudometric $\rho$ on $X$ is {\it dynamically generating} \cite[Sect.\ 4]{Li} if
for any distinct points $x,y\in X$ one has $\rho(sx, sy)>0$ for some $s\in G$. The next lemma says that for any dynamically generating continuous pseudometric
$\rho$, one can construct a compatible metric $\rho'$ on $X$ such that $\mdim_{\Sigma, \rM} (X, \rho )=\mdim_{\Sigma, \rM}(X, \rho')$. This is the analogue of the Kolmogorov-Sinai theorem that for measure-preserving actions of $\Zb$ the entropy of any generating finite partition is equal to the entropy of the action. In some examples $X$ is naturally a closed invariant subset of $Y^G$ equipped with the shift action of $G$ for some compact metrizable space $Y$ which is much simpler than $X$. In such case one can take a compatible metric $\rho$ on $Y$ and think of it as a dynamically generating continuous pseudometric on $X$ via the coordinate map $X\rightarrow Y$ at $e_G$. Then the calculation using $\rho$ is much simpler than that using $\rho'$ (see for example Section~\ref{S-shifts} and \cite{BL, KL11, KerLi10amenable}).

\begin{lemma} \label{L-allow pseudometric}
Let $\rho$ be a dynamically generating continuous pseudometric on $X$. Enumerate the elements of $G$ as $s_1, s_2, \dots$.
Define $\rho'$ by $\rho'(x, y)=\sum_{n=1}^{\infty}\frac{1}{2^n}\rho(s_nx, s_ny)$ for all $x, y\in X$. Then $\rho'$ is a compatible metric on $X$.
Furthermore, if $e_G=s_m$, then for any $\varepsilon>0$ one has
$$ h_{\Sigma ,\infty}^{4\varepsilon} (\rho' )\le h_{\Sigma ,\infty}^\varepsilon (\rho )\le h_{\Sigma ,\infty}^{\varepsilon/2^m} (\rho' ).$$
In particular,
$\mdim_{\Sigma, \rM} (X, \rho )=\mdim_{\Sigma, \rM}(X, \rho')$.
\end{lemma}
\begin{proof} Clearly $\rho'$ is a continuous pseudometric on $X$. Since $\rho$ is dynamically generating, $\rho'$ separates the points of $X$. Thus $\rho'$ is a compatible metric on $X$. Let $\varepsilon>0$.

We show first
$h_{\Sigma ,\infty}^\varepsilon (\rho )\le h_{\Sigma ,\infty}^{\varepsilon/2^m} (\rho' )$.
Let $F$ be a finite subset of $G$ containing $e_G$ and $\delta>0$. Take $k\in \Nb$ with
$2^{-k}\diam(X, \rho)<\delta/2$. Set $F'=\bigcup_{n=1}^ks_nF$ and take $1>\delta'>0$ to be small which we shall fix in a moment.

Let $\sigma$ be a map from $G$ to  $\Sym(d)$ for some $d\in \Nb$ which is a good enough sofic approximation for $G$.
We claim that $\Map (\rho ,F',\delta' ,\sigma )\subseteq \Map (\rho' ,F,\delta ,\sigma)$.
Let $\varphi\in \Map (\rho ,F',\delta' ,\sigma )$.
By Lemma~\ref{L-almost} one has
$$|\cW|\ge (1-\delta'|F'|)d,$$
for
$$\cW:=\{a\in [d]: \max_{s\in F'}\rho(\varphi\circ \sigma_{s}(a), \alpha_s\circ \varphi(a))\le \sqrt{\delta'}\}.$$
Set $\cR=\cW\cap \bigcap_{t\in F}\sigma_{t}^{-1}(\cW)$. Then $|\cR|\ge (1-\delta'|F'|(1+|F|))d$.
Also set
$$\cQ=\{a\in [d]: \sigma_{s_n}\circ \sigma_{t}(a)=\sigma_{s_nt}(a) \mbox{ for all } 1\le n \le k \mbox{ and } t\in F \}.$$
For any $a\in \cR\cap \cQ$ and $t\in F$, since $a, \sigma_{t}(a)\in \cW$ and $s_n, s_nt\in F'$ for all $1\le n\le k$, we have
\begin{eqnarray*}
& &\rho'(\varphi \circ \sigma_{t}(a), \alpha_t\circ \varphi(a))\\
&\le& 2^{-k}\diam(X, \rho)+\sum_{n=1}^k\frac{1}{2^n}\rho(\alpha_{s_n}\circ \varphi \circ \sigma_{t}(a), \alpha_{s_n}\circ \alpha_t\circ \varphi(a))\\
&\le& \delta/2+\sum_{n=1}^k \frac{1}{2^n}\bigg(\rho(\alpha_{s_n}\circ \varphi \circ \sigma_{t}(a), \varphi\circ \sigma_{s_n}\circ \sigma_{t}(a))+\rho(\varphi\circ \sigma_{s_nt}(a), \alpha_{s_nt}\circ \varphi(a))\bigg)\\
&\le& \delta/2+\sum_{n=1}^k \frac{1}{2^n}\cdot 2\sqrt{\delta'}\le \delta/2+2\sqrt{\delta'}.
\end{eqnarray*}
When $\sigma$ is a good enough sofic approximation for $G$, one has $|\cQ|\ge (1-\delta'|F'|)d$ and hence for any $t\in F$,
\begin{eqnarray*}
(\rho'_2(\varphi \circ \sigma_{t}, \alpha_t\circ \varphi))^2&\le& \frac{1}{d}(|\cR\cap \cQ|(\delta/2+2\sqrt{\delta'})^2+(d-|\cR\cap \cQ|)(\diam(X, \rho'))^2)\\
&\le& (\delta/2+2\sqrt{\delta'})^2+\delta'|F'|(2+|F|)(\diam(X, \rho'))^2<\delta^2,
\end{eqnarray*}
when $\delta'$ is small enough independent of $\sigma$ and $\varphi$. Therefore $\varphi\in \Map (\rho' ,F,\delta ,\sigma)$. This proves the claim.

Note that $\frac{1}{2^m}\rho_\infty\le \rho'_\infty$ on $\Map (\rho ,F',\delta' ,\sigma )$. Thus
$$N_\varepsilon(\Map (\rho ,F',\delta' ,\sigma ), \rho_\infty)\le N_{\varepsilon/2^m}(\Map (\rho ,F',\delta' ,\sigma ), \rho'_\infty)
\le N_{\varepsilon/2^m}(\Map (\rho' ,F,\delta ,\sigma ), \rho'_\infty),$$
when $\sigma$ is a good enough sofic approximation for $G$.
It follows that $h_{\Sigma ,\infty}^\varepsilon (\rho ,F', \delta' )\le h_{\Sigma ,\infty}^{\varepsilon/2^m} (\rho' ,F, \delta )$, and
hence $ h_{\Sigma ,\infty}^\varepsilon (\rho )\le h_{\Sigma ,\infty}^{\varepsilon/2^m} (\rho' )$ as desired.

Next we show
$h_{\Sigma ,\infty}^{4\varepsilon} (\rho' )\le h_{\Sigma ,\infty}^\varepsilon (\rho )$.
It suffices to show $h_{\Sigma ,\infty}^{4\varepsilon} (\rho' )\le h_{\Sigma ,\infty}^\varepsilon (\rho )+\theta$ for every $\theta>0$.
Take $k\in \Nb$ with
$2^{-k}\diam(X, \rho)<\varepsilon/2$.

Let $F$ be a finite subset of $G$ containing $\{s_1, \dots, s_k\}$ and $\delta>0$ be sufficiently small  which we shall specify in a moment.
Set $\delta'=\delta/2^m$.
Let $\sigma$ be a map from $G$ to $\Sym(d)$ for some sufficiently large $d\in \Nb$.
Note that $\frac{1}{2^m}\rho_2(\varphi, \psi)\le \rho'_2(\varphi, \psi)$ for all maps $\varphi, \psi: [d]\rightarrow X$.
Thus $\Map (\rho ,F,\delta ,\sigma )\supseteq \Map (\rho' ,F,\delta' ,\sigma )$.

Let $\sE$ be a $(\rho'_\infty, 4\varepsilon)$-separated subset of $\Map (\rho' ,F,\delta' ,\sigma )$ with $|\sE|=N_{4\varepsilon}(\Map (\rho' ,F,\delta' ,\sigma), \rho'_\infty)$. For each $\varphi\in \sE$ denote by $\cW_\varphi$ the set of $a\in [d]$ satisfying
$\rho(\alpha_s\circ \varphi(a), \varphi\circ \sigma_s(a))\le \sqrt{\delta}$ for all $s\in F$. By Lemma~\ref{L-almost} one has $|\cW_\varphi|\ge (1-|F|\delta)d$.
Take $\delta$ to be small enough so that $|F|\delta<1/2$.
Stirling's approximation says that $\frac{t!}{(2\pi t)^{1/2}(t/e)^t)}\to 1$ as $t\to +\infty$.
The number of subsets of $[d]$ of cardinality at most $|F|\delta d$ is equal to $\sum_{j=0}^{\lfloor |F|\delta d\rfloor}\binom{d}{j}$,
which is at most $|F|\delta d\binom{d}{|F|\delta d}$ (here we use $|F|\delta<1/2$), which by Stirling's approximation is less than $\exp(\beta d)$
for some $\beta > 0$ depending on $\delta$ and $|F|$
but not on $d$ when $d$ is sufficiently large with $\beta\to 0$ as $\delta\to 0$ for a fixed $|F|$.
Take $\delta$ to be small enough such that $\beta<\theta/2$. Then, when $d$ is sufficiently large, we can find a subset
$\sF$ of $\sE$ with $|\sF|\exp(\beta d)\ge |\sE|$ such that $\cW_\varphi$ is the same, say $\cW$, for every $\varphi \in \sF$.

Let $\varphi \in \sF$. Let us estimate how many elements in $\sF$ are in the open ball $B_\varphi:=\{\psi\in X^{[d]}: \rho_\infty(\varphi, \psi)<\varepsilon\}$. Let $\psi\in \sF\cap B_\varphi$. For any $a\in \cW$ and $s\in F$, we have
\begin{eqnarray*}
\rho(s\varphi(a), s\psi(a))&\le& \rho(s\varphi(a), \varphi(sa))+\rho(\varphi(sa), \psi(sa))+\rho(\psi(sa), s\psi(a))\\
&\le & \sqrt{\delta}+\varepsilon+\sqrt{\delta}\le \frac{3}{2}\varepsilon,
\end{eqnarray*}
when $\delta$ is taken to be small enough.
It follows that for any $a\in \cW$ we have
\begin{eqnarray*}
 \rho'(\varphi(a), \psi(a))&\le& 2^{-k}\diam(X, \rho)+\sum_{n=1}^k2^{-n}\rho(s_n\varphi(a), s_n\psi(a))\\
 &<& \frac{1}{2}\varepsilon+\frac{3}{2}\varepsilon=2\varepsilon.
\end{eqnarray*}
Then $\rho'_\infty(\varphi|_{\cW}, \psi|_{\cW})< 2\varepsilon$.

Let $Y$ be a maximal $(\rho', 2\varepsilon)$-separated subset of $X$.
Set $\cW^c=[d]\setminus \cW$.
For each $\psi\in \sF\cap B_\varphi$, there exists some $f_\psi\in Y^{\cW^c}$ with
$\rho'_\infty(\psi|_{\cW^c}, f_\psi)<2\varepsilon$. Then we can find a subset $\sA$ of $\sF\cap B_\varphi$ with $|Y|^{|\cW^c|}|\sA|\ge |\sF\cap B_\varphi|$ such that
$f_\psi$ is the same, say $f$, for every $\psi\in \sA$. For any $\psi, \psi'\in \sA$, we have
$$ \rho'_{\infty}(\psi|_{\cW^c}, \psi'|_{\cW^c})\le \rho'_\infty(\psi|_{\cW^c}, f)+\rho'_\infty(f, \psi'|_{\cW^c})<4\varepsilon,$$
and
$$ \rho'_{\infty}(\psi|_{\cW}, \psi'|_{\cW})\le \rho'_\infty(\psi|_{\cW}, \varphi|_\cW)+\rho'_\infty(\varphi|_\cW, \psi'|_{\cW})<4\varepsilon,$$
and hence $\rho'_\infty(\psi, \psi')< 4\varepsilon$.  Since $\sA$ is $(\rho'_\infty, 4\varepsilon)$-separated, we get $\psi=\psi'$.
Therefore $|\sA|\le 1$, and hence
\begin{align*}
|\sF\cap B_\varphi|\le |Y|^{|\cW^c|}|\sA|\le |Y|^{|F|\delta d}.
\end{align*}

Let $\sB$ be a maximal $(\rho_\infty, \varepsilon)$-separated subset of $\sF$. Then $\sF=\bigcup_{\varphi \in \sB}(\sF\cap B_\varphi)$.
Thus
\begin{align*}
N_{4\varepsilon}(\Map (\rho' ,F,\delta' ,\sigma ), \rho'_\infty)&=|\sE|\le \exp(\beta d)|\sF|\le \exp(\theta d/2)|\sB||Y|^{|F|\delta d}\\
&\le \exp(\theta d/2)|Y|^{|F|\delta d} N_\varepsilon(\Map (\rho, F, \delta, \sigma), \rho_\infty)\\
&\le \exp(\theta d)N_\varepsilon(\Map (\rho, F, \delta, \sigma), \rho_\infty)
\end{align*}
when we take $\delta$ to be small enough, where in the third inequality we used $\sB\subseteq \Map (\rho' ,F,\delta' ,\sigma )\subseteq \Map (\rho, F, \delta, \sigma)$ to conclude $|\sB|\le N_\varepsilon(\Map (\rho, F, \delta, \sigma), \rho_\infty)$. Therefore $h_{\Sigma ,\infty}^{4\varepsilon} (\rho' ,F, \delta' )\le h_{\Sigma ,\infty}^\varepsilon (\rho ,F, \delta )+\theta$. It follows that $h_{\Sigma ,\infty}^{4\varepsilon} (\rho' )\le h_{\Sigma ,\infty}^\varepsilon (\rho )+\theta$  as desired.
\end{proof}

Note that $\mdim_{\Sigma, \rM} (X)$ was defined as the infimum over all compatible metrics $\rho$ of $ \mdim_{\Sigma, \rM}(X, \rho)$. However
by Lemma~\ref{L-allow pseudometric} we get:

\begin{proposition} \label{P-allow pseudometric}
One has
\[
\mdim_{\Sigma, \rM} (X) = \inf_{\rho}\mdim_{\Sigma, \rM}(X, \rho)
\]
for $\rho$ ranging over dynamically generating continuous pseudometrics on $X$.
\end{proposition}

The following is the analogue of Proposition~\ref{P-subspace top}.

\begin{proposition} \label{P-subspace metric}
Let $G$ act continuously on a compact metrizable space $X$. Let $Y$ be a closed $G$-invariant subset of $X$. Then $\mdim_{\Sigma, \rM}(Y)\le \mdim_{\Sigma, \rM}(X)$.
\end{proposition}
\begin{proof}
Let $\rho$ be a compatible metric on $X$. Then $\rho$ restricts to a compatible metric $\rho'$ on $Y$.

Note that $\Map(\rho', F, \delta, \sigma)\subseteq \Map(\rho, F, \delta, \sigma)$ for any nonempty finite subset $F$ of $G$, any $\delta>0$, and any
map $\sigma$ from $G$ to $\Sym(d)$ for some $d\in \Nb$. Furthermore, the restriction of $\rho_\infty$ on $\Map(\rho', F, \delta, \sigma)$ is exactly $\rho'_\infty$. Thus $N_\varepsilon(\Map(\rho', F, \delta, \sigma), \rho'_\infty)\le N_\varepsilon(\Map(\rho, F, \delta, \sigma), \rho_\infty)$ for any $\varepsilon>0$. It follows that
$\mdim_{\Sigma, \rM}(Y)\le \mdim_{\Sigma, \rM}(Y, \rho')\le \mdim_{\Sigma, \rM}(X, \rho)$.
Since $\rho$ is an arbitrary compatible metric on $X$, we get $\mdim_{\Sigma, \rM}(Y)\le \mdim_{\Sigma, \rM}(X)$.
\end{proof}

\section{Sofic Metric Mean Dimension for Amenable Groups} \label{S-metric amenable}

In this section we show that the sofic metric mean dimension extends the metric mean dimension for actions of countable amenable groups:

\begin{theorem}\label{T-amenable}
Let a countable (discrete) amenable group $G$ act continuously on a compact metrizable space $X$.
Let $\Sigma$ be a sofic approximation sequence for $G$.
Then
\[
\mdim_{\Sigma, \rM} (X,\rho) = \mdim_\rM (X,\rho)
\]
for every continuous pseudometric $\rho$ on $X$. In particular,
\[ \mdim_{\Sigma, \rM}(X)=\mdim_\rM (X).
\]
\end{theorem}

Theorem~\ref{T-amenable} follows directly from Lemmas~\ref{L-amenable lower bound} and \ref{L-amenable upper bound} below.
The proof of Lemma~\ref{L-amenable lower bound} is similar to that of Lemma~\ref{L-top mean dim lower bound}.

\begin{lemma}\label{L-amenable lower bound}
Let a countable amenable group $G$ act continuously on a compact metrizable space $X$. Let $\rho$ be a continuous pseudometric on $X$.
Then for any $\varepsilon>0$ we have $ h^{\varepsilon}_{\Sigma ,\infty} (\rho)\ge S(X, 2\varepsilon, \rho)$. In particular,
$\mdim_{\Sigma, \rM} (X, \rho)\ge \mdim_\rM(X, \rho)$.
\end{lemma}
\begin{proof}
It suffices to show that $h^\varepsilon_{\Sigma ,\infty} (\rho ) \geq S(X, 2\varepsilon, \rho) - 2\theta$ for every $\theta>0$.

Take a nonempty finite subset $K$ of $G$ and $\varepsilon'>0$ such that
for any nonempty finite subset $F'$ of $G$ satisfying $|KF'\setminus F'|<\varepsilon'|F'|$, one has
\begin{align} \label{E-amenable lower bound}
 \frac{1}{|F'|}\log N_{\varepsilon}(X, \rho_{F'})\overset{\eqref{E-separated vs spanning}}\ge \frac{1}{|F'|}\log \min_{\mesh(\cU, \rho_{F'})<2\varepsilon} |\cU|\ge S(X, 2\varepsilon, \rho)-\theta.
\end{align}

Let $F$ be a nonempty finite subset of $G$ and $\delta > 0$. Let $\sigma$ be a map from $G$ to $\Sym (d)$
for some $d\in\Nb$.
Now it suffices to show that if $\sigma$ is a good enough sofic approximation then
\begin{align*}
\frac{1}{d} \log N_{\varepsilon} (\Map (\rho ,F,\delta ,\sigma ),\rho_{\infty} )\ge S(X, 2\varepsilon, \rho) - 2\theta .
\end{align*}

Take $\delta'>0$ such that $\sqrt{\delta'}\diam(X, \rho)<\delta/2$ and
$(1-\delta' )(S(X, 2\varepsilon, \rho) - \theta ) \ge S(X, 2\varepsilon, \rho) - 2\theta$.
By Lemma~\ref{L-Rokhlin} there are an $\ell\in \Nb$
and nonempty finite subsets $F_1, \dots, F_\ell$ of $G$
satisfying $|(K\cup F)F_k\setminus F_k|<\min(\varepsilon', \delta') |F_k|$ for all $k=1, \dots, \ell$
such that for every map $\sigma : G\to\Sym(d)$ for some $d\in \Nb$ which is a good
enough sofic approximation for $G$ and every $\cW\subseteq [d]$ with $|\cW|\ge (1-\delta'/2)d$ there exist  $\cC_1, \dots, \cC_\ell\subseteq \cW$ satisfying the following:
\begin{enumerate}
\item for every $k=1, \dots, \ell$, the map $(s, c)\mapsto \sigma_s(c)$ from $F_k\times \cC_k$ to $\sigma(F_k)\cC_k$ is bijective,

\item the sets $\sigma(F_1)\cC_1, \dots, \sigma(F_\ell)\cC_\ell $ are pairwise disjoint and $|\bigcup_{k=1}^\ell\sigma(F_k)\cC_k|\ge (1-\delta')d$.
\end{enumerate}

Let $\sigma: G\rightarrow \Sym(d)$ for some $d\in \Nb$ be a good enough sofic approximation for $G$ such that $|\cW|\ge (1-\delta'/2)d$ for
$$ \cW:=\{a\in [d]: \sigma_t\sigma_s(a)=\sigma_{ts}(a) \mbox{ for all } t\in F, s\in \bigcup_{k=1}^\ell F_k\}.$$
Then we have $\cC_1, \dots, \cC_\ell$ as above.

For every $k\in\{ 1,\dots, \ell\}$ pick an $\varepsilon$-separated set
$E_k \subseteq X$ with respect to $\rho_{F_k}$ of maximal cardinality.
Then
$$ \frac{1}{|F_k|}\log |E_k|=\frac{1}{|F_k|}\log N_\varepsilon(X, \rho_{F_k})\overset{\eqref{E-amenable lower bound}}\ge S(X, 2\varepsilon, \rho)-\theta.$$
For each $h = (h_k )_{k=1}^\ell \in\prod_{k=1}^\ell (E_k )^{\cC_k}$
take a map $\varphi_h : [d]\rightarrow X$ such that
\[
\varphi_h(s c) = s(h_k (c))
\]
for all $k\in \{ 1,\dots ,\ell \}$, $c\in \cC_k$, and $s\in F_k$.
By Lemma~\ref{L-construct map} one has
$\varphi_h \in\Map (\rho , F ,\delta , \sigma )$.

Now if $h = (h_k )_{k=1}^\ell$ and $h' = (h_k' )_{k=1}^\ell$ are distinct elements of
$\prod_{k=1}^\ell (E_k )^{\cC_k}$, then $h_k (c) \neq h_k' (c)$ for some $k\in \{ 1,\dots ,\ell \}$ and $c\in \cC_k$.
Since $h_k (c)$ and $h_k' (c)$ are distinct points in $E_k$ which is $\varepsilon$-separated with respect to $\rho_{F_k}$,
$h_k (c)$ and $h_k' (c)$ are $\varepsilon$-separated with respect to $\rho_{F_k}$, and thus
we have
$\rho_{\infty} (\varphi_h ,\varphi_{h'} ) \ge \varepsilon$. Therefore
\begin{align*}
\frac1d \log N_{\varepsilon} (\Map (\rho ,F,\delta ,\sigma ),\rho_{\infty} )
&\geq \frac1d \sum_{k=1}^\ell |\cC_k | \log |E_k | \\
&\geq \frac1d \sum_{k=1}^\ell |\cC_k | |F_k | (S(X, 2\varepsilon, \rho) - \theta ) \\
&\geq (1-\delta' )(S(X, 2\varepsilon, \rho) - \theta ) \\
&\geq S(X, 2\varepsilon, \rho) - 2\theta ,
\end{align*}
as desired.
\end{proof}

The proof of Lemma~\ref{L-amenable upper bound} is similar to that of Lemma~\ref{L-top mean dim upper bound}, but considerably easier.

\begin{lemma}\label{L-amenable upper bound}
Let a countable amenable group $G$ act continuously on a compact metrizable space $X$. Let $\rho$ be a continuous pseudometric on $X$.
Then for any $\varepsilon>0$ we have $ h^{\varepsilon}_{\Sigma ,\infty} (\rho)\le S(X, \varepsilon/4, \rho)$. In particular,
$\mdim_{\Sigma, \rM} (X, \rho)\le \mdim_\rM(X, \rho)$.
\end{lemma}

\begin{proof}
Let $\varepsilon> 0$.
It suffices to show that
$h^{\varepsilon}_{\Sigma ,\infty} (\rho)\le S(X, \varepsilon/4, \rho) + 3\theta$ for every $\theta >0$.

Take a nonempty finite subset $K$ of $G$ and $\delta'>0$ such that
$\min_{\mesh(\cU, \rho_{F'})<\varepsilon/4} |\cU|<\exp((S(X, \varepsilon/4, \rho) + \theta)|F'|)$ for every nonempty finite subset
$F'$ of $G$ satisfying $|KF'\setminus F'|<\delta'|F'|$.
Take an $\eta\in (0,1)$ such that $(N_{\varepsilon/2}(X, \rho ))^{2\eta} \leq \exp(\theta)$.

By Lemma~\ref{L-Rokhlin} there are an $\ell\in \Nb$
and nonempty finite subsets $F_1, \dots, F_\ell$ of $G$
satisfying $|KF_k\setminus F_k|<\delta' |F_k|$ for all $k=1, \dots, \ell$
such that for every map $\sigma : G\to\Sym(d)$ for some $d\in \Nb$ which is a good
enough sofic approximation for $G$ and every $\cW\subseteq [d]$ with $|\cW|\ge (1-\eta)d$ there exist  $\cC_1, \dots, \cC_\ell\subseteq \cW$ satisfying the following:
\begin{enumerate}
\item for every $k=1, \dots, \ell$, the map $(s, c)\mapsto \sigma_s(c)$ from $F_k\times \cC_k$ to $\sigma(F_k)\cC_k$ is bijective,

\item the sets $\sigma(F_1)\cC_1, \dots, \sigma(F_\ell)\cC_\ell $ are pairwise disjoint and $|\bigcup_{k=1}^\ell\sigma(F_k)\cC_k|\ge (1-2\eta)d$.
\end{enumerate}
Then
\begin{align} \label{E-top upper}
N_{\varepsilon/4}(X, \rho_{F_k})\overset{\eqref{E-separated vs spanning}}\le \min_{\mesh(\cU, \rho_{F_k})<\varepsilon/4} |\cU|<\exp((S(X, \varepsilon/4, \rho) + \theta)|F_k|)
\end{align}
for every $k=1, \dots, \ell$.

Set $F=\bigcup_{k=1}^\ell F_k$. Let $\delta > 0$ be a small positive number which we will determine in a moment.
Let $\sigma$ be a map from $G$ to $\Sym (d)$ for some sufficiently large $d\in\Nb$ which is a good enough sofic approximation for $G$.
We will show that $N_{\varepsilon}(\Map(\rho, F, \delta, \sigma), \rho_{\infty})\le \exp(S(X, \varepsilon/4, \rho) + 3\theta)d$,
which will complete the proof since we can then conclude that
$h^{\varepsilon}_{\Sigma ,\infty}(\rho, F, \delta)\le S(X, \varepsilon/4, \rho) + 3\theta$
and hence $h^{\varepsilon}_{\Sigma ,\infty}(\rho)\le S(X, \varepsilon/4, \rho) + 3\theta$.

For every $\varphi\in\Map (\rho, F, \delta, \sigma)$,
by Lemma~\ref{L-almost}
the set $\cW_{\varphi}$
of all $a\in [d]$ satisfying
\[
\rho(\varphi (sa), s\varphi(a)) \le \sqrt{\delta}
\]
for all $s\in F$ has cardinality at least $(1-|F|\delta) d$.

For each $\cW\subseteq [d]$ we
define on the set of maps from $[d]$ to $X$ the pseudometric
\[
\rho_{\cW,\infty} (\varphi , \psi ) = \rho_{\infty} (\varphi|_\cW,  \psi|_\cW).
\]

Take a $(\rho_\infty, \varepsilon )$-separated subset $\sE$ of $\Map (\rho, F, \delta, \sigma)$
of maximal cardinality.

Set $n = |F|$. Stirling's approximation says that $\frac{t!}{(2\pi t)^{1/2}(t/e)^t)}\to 1$ as $t\to +\infty$.
When $n\delta <1/2$,
the number of subsets of $[d]$ of cardinality no greater than $n\delta d$
is equal to $\sum_{j=0}^{\lfloor n\delta d \rfloor} \binom{d}{j}$, which is at most
$n\delta d \binom{d}{n\delta d}$,
which by Stirling's approximation is less than $\exp(\beta d)$
for some $\beta > 0$ depending on $\delta$ and $n$
but not on $d$ when $d$ is sufficiently large with $\beta\to 0$ as $\delta\to 0$ for a fixed $n$.
Thus when $\delta$ is small enough and $d$ is large enough, there is a subset $\sF$ of $\sE$ with $\exp(\theta d)|\sF|\ge |\sE|$
such that the set $\cW_{\varphi}$ is the same, say $\cW$, for every $\varphi\in \sF$,
and $|\cW|/d>1-\eta$.
Then we have $\cC_1, \dots, \cC_\ell\subseteq \cW$ as above.

Let $1\le k\le \ell$ and $c\in \cC_k$.
Let $\sD_{k, c}$ be a maximal $(\varepsilon/2)$-separated subset of $\sF$ with respect to $\rho_{\sigma(F_k)c ,\infty}$.
Then $\sD_{k, c}$ is a $(\rho_{\sigma(F_k)c ,\infty}, \varepsilon/2)$-spanning subset of $\sF$ in the sense that for any $\varphi\in \sF$, there exists some
$\psi\in \sD_{k, c}$ with $\rho_{\sigma(F_k)c ,\infty}(\varphi, \psi)<\varepsilon/2$.
We will show that
$|\sD_{k, c}|\le \exp((S(X, \varepsilon/4, \rho) + \theta)|F_k|)$ when $\delta$ is small enough.
For any two distinct elements $\varphi$ and $\psi$ of $\sD_{k, c}$ we have, for every $s\in F_k$, since $c\in \cW_{\varphi}\cap \cW_{\psi}$,
\begin{align*}
\rho(s\varphi(c), s\psi(c))&\ge \rho(\varphi(sc), \psi(sc))-\rho(s\varphi(c), \varphi(sc))-\rho(s\psi(c), \psi(sc))\\
&\ge  \rho(\varphi(sc), \psi(sc))-2\sqrt{\delta},
\end{align*}
and hence
\begin{align*}
\rho_{F_k}(\varphi(c), \psi(c))&=\max_{s\in F_k}\rho(s\varphi(c), s\psi(c))\ge \max_{s\in F_k}\rho(\varphi(sc), \psi(sc))-2\sqrt{\delta}>\varepsilon/2-\varepsilon/4=\varepsilon/4,
\end{align*}
granted that $\delta$ is taken small enough.
Thus $\{\varphi(c) : \varphi\in \sD_{k, c} \}$ is a $(\rho_{F_k}, \varepsilon/4 )$-separated subset of $X$ of cardinality $|\sD_{k, c}|$,
so that
\begin{align*}
|\sD_{k, c}|\le N_{\varepsilon/4}(X, \rho_{F_k})
\overset{\eqref{E-top upper}}\le \exp((S(X, \varepsilon/4, \rho) + \theta)|F_k|),
\end{align*}
as we wished to show.

Set
\[
\cZ = [d] \setminus \bigcup_{k=1}^\ell\sigma(F_k)\cC_k,
\]
and take a maximal $(\rho, \varepsilon/2)$-separated subset $Y$ of $X$. Then $Y^\cZ$ is a $(\rho_\infty, \varepsilon/2)$-spanning subset of $X^\cZ$ in the sense that for any $\varphi\in X^\cZ$ there is some $\psi\in Y^\cZ$ with $\rho_\infty(\varphi, \psi)<\varepsilon/2$.
We have
\[
|Y^\cZ | \leq (N_{\varepsilon/2}(X, \rho ))^{|\cZ|} \leq (N_{\varepsilon/2}(X, \rho ))^{2\eta d}.
\]

Write $\sA$ for the set of all maps
$\varphi : [d]\rightarrow X$ such that $\varphi|_\cZ\in Y^\cZ$ and
$\varphi|_{\sigma (F_k)c}\in \sD_{k,c}|_{\sigma (F_k)c}$ for all $1\le k\le \ell$ and $c\in \cC_k$.
Then, by our choice of $\eta$,
\begin{align*}
|\sA | &= |Y^\cZ | \prod_{k=1}^\ell \prod_{c\in \cC_k} |\sD_{k,c} |
\le (N_{\varepsilon/2}(X, \rho ))^{2\eta d} \exp\bigg(\sum_{k=1}^\ell\sum_{c\in \cC_k}(S(X, \varepsilon/4, \rho) + \theta)|F_k|\bigg)\\
&= (N_{\varepsilon/2}(X, \rho ))^{2\eta d}\exp\bigg((S(X, \varepsilon/4, \rho) + \theta)\sum_{k=1}^\ell|F_k| |C_k|\bigg)\\
&\le \exp(\theta d)
\exp((S(X, \varepsilon/4, \rho) + \theta)d)=\exp((S(X, \varepsilon/4, \rho) + 2\theta)d).
\end{align*}

Now since every element of $\sF$ lies within $\rho_{\infty}$-distance $\varepsilon/2$ to an element
of $\sA$ and $\sF$ is $\varepsilon$-separated with respect to $\rho_{\infty}$,
the cardinality of $\sF$ is at most that of $\sA$.
Therefore
\begin{align*}
N_{\varepsilon}(\Map (\rho, F, \delta, \sigma), \rho_\infty)&=|\sE|\le \exp(\theta d)|\sF|\le \exp(\theta d)|\sA|\\
&\le \exp(\theta d)\exp((S(X, \varepsilon/4, \rho) + 2\theta)d)\\
&=\exp((S(X, \varepsilon/4, \rho) + 3\theta)d),
\end{align*}
as desired.
\end{proof}

\section{Comparison of Sofic Mean Dimensions} \label{S-comparison}

In this section we prove the following relation between sofic mean dimensions:

\begin{theorem} \label{T-top vs metric}
Let a countable sofic group $G$ act continuously on  a compact metrizable space $X$. Let $\Sigma$ be a sofic approximation sequence of $G$.  Then
$$ \mdim_\Sigma(X)\le \mdim_{\Sigma, \rM}(X).$$
\end{theorem}

The amenable group case of Theorem~\ref{T-top vs metric} was proved by Lidenstrauss and Weiss \cite[Theorem 4.2]{LW}. We adapt their proof to our situation, by
replacing the set $\{(sx)_{s\in H}: x\in X\}\subseteq X^H$ for an approximately left invariant finite subset $H$ of $G$ in their proof (in their case $G=\Zb$ and $H=\{0, \dots, N-1\}$) with the set $\Map(\rho, F, \delta, \sigma)$.

\begin{lemma} \label{L-Lipschitz partion of unity}
Let $\cU$ be a finite open cover of $X$ and $\rho$ be a compatible metric on $X$.
Then there exist Lipschitz functions $f_U: X\rightarrow [0,1]$ vanishing on $X\setminus U$ for each $U\in \cU$ such
that $\max_{U\in \cU}f_U(x)=1$ for every $x\in X$.
\end{lemma}
\begin{proof}
If some elements of $\cU$ are equal to the whole space $X$, we may set $f_U$ to be the constant function $1$ for the elements $U$ equal to $X$ and the constant function $0$ for those elements $U$ not equal to $X$. Thus we mat assume that all elements of $\cU$ do not equal $X$.
Note that for each $U\in \cU$, the function $\rho(\cdot, X\setminus U): X\rightarrow [0, +\infty)$
is Lipschitz and vanishes on $X\setminus U$.  Furthermore,
$ \sum_{V\in \cU} \rho(x, X\setminus V)>0$ for every $x\in X$. Define
$g_U, f_U: X\rightarrow [0, 1]$ by
$$g_U(x)=\frac{\rho(x, X\setminus U)}{\sum_{V\in \cU} \rho(x, X\setminus V)},$$
and
$$ f_U(x)=|\cU|\min(g_U(x), 1/|\cU|).$$
Then $g_U$ is Lipschitz and vanishes on $X\setminus U$, and hence so is $f_U$. Furthermore, for each $x\in X$, one has
$\sum_{U\in \cU}g_U(x)=1$ and thus $g_U(x)\ge 1/|\cU|$ for some $U\in \cU$. It follows that $f_U(x)=1$ for some $U\in \cU$.
\end{proof}

Let $\rho, \cU$ and $f_U$ be as in Lemma~\ref{L-Lipschitz partion of unity}.
Let $\sigma$ be a map from $G$ to $\Sym(d)$ for some $d\in \Nb$. We define  a  continuous map $\Phi_d: X^{[d]}\rightarrow [0, 1]^{[d]\times \cU}$
by $\Phi_d(\varphi)_{a, U}=f_U(\varphi(a))$ for $a\in [d]$ and $U\in \cU$.

\begin{lemma} \label{L-small image}
Let $\theta>0$, and set $D=\mdim_{\Sigma, \rM}(X, \rho)$. Then there exist a nonempty finite subset $F$ of $G$, $\delta>0$ and $M>0$ such that for any $i\in \Nb$ with $i\ge M$, there exists
$\xi\in (0, 1)^{[d_i]\times \cU}$ such that for any $S\subseteq [d_i]\times \cU$ with $|S|\ge (D+\theta)d_i$,
one has
$$ \xi|_S\not \in \Phi_{d_i}(\Map(\rho, F, \delta, \sigma_i))|_S.$$
\end{lemma}
\begin{proof} Denote by $C$ the maximum of the Lipschitz constants of $f_U$ for all $U\in \cU$.
Then
$$ \|\Phi_d(\varphi)-\Phi_d(\psi)\|_\infty\le C \rho_\infty(\varphi, \psi)$$
for all $d\in \Nb$ and $\varphi, \psi\in X^{[d]}$.

Take $\varepsilon>0$ small enough, which we shall determine in a moment, satisfying
$$ \frac{h^\varepsilon_{\Sigma, \infty}(\rho)}{|\log \varepsilon|}<D+\theta/3.$$
Then we can find a nonempty finite subset $F$ of $G$, $\delta>0$, and $M>0$ such that
for every $i\in \Nb$ with $i>M$, we have
$$ \frac{1}{d_i|\log \varepsilon|}\log N_\varepsilon(\Map(\rho, F, \delta, \sigma_i), \rho_\infty)\le D+\theta/2, $$
that is,
$$ N_\varepsilon(\Map(\rho, F, \delta, \sigma_i), \rho_\infty)\le \varepsilon^{-(D+\theta/2)d_i}.$$
It follows that we can cover $\Phi_{d_i}(\Map(\rho, F, \delta, \sigma_i))$ using $\varepsilon^{-(D+\theta/2)d_i}$ open balls in the $\|\cdot \|_\infty$ norm with radius $\varepsilon C$.

Denote by $\mu$ the Lebesgue measure on $[0, 1]^{[d_i]\times \cU}$. For each $S\subseteq [d_i]\times \cU$, the set
$\Phi_{d_i}(\Map(\rho, F, \delta, \sigma_i))|_S\subseteq [0, 1]^S$ can be covered using $\varepsilon^{-(D+\theta/2)d_i}$ open balls in the $\|\cdot \|_\infty$ norm with radius $\varepsilon C$, and hence has Lebesgue measure at most $\varepsilon^{-(D+\theta/2)d_i}(2\varepsilon C)^{|S|}$.
Thus,
$$\mu(\{\xi\in [0, 1]^{[d_i]\times \cU}: \xi|_S\in \Phi_{d_i}(\Map(\rho, F, \delta, \sigma_i))|_S\})\le \varepsilon^{-(D+\theta/2)d_i}(2\varepsilon C)^{|S|}.$$
Therefore, the set of $\xi\in [0,1]^{[d_i]\times \cU}$ satisfying $\xi|_S\in \Phi_{d_i}(\Map(\rho, F, \delta, \sigma_i))|_S$ for some
$S\subseteq [d_i]\times \cU$ with $|S|\ge (D+\theta)d_i$ has Lebesgue measure at most
\begin{align*}
2^{|[d_i]\times \cU|}\varepsilon^{-(D+\theta/2)d_i}(2\varepsilon C)^{(D+\theta)d_i}= (2^{|\cU|}\varepsilon^{\theta/2}(2C)^{D+\theta})^{d_i}<1,
\end{align*}
when we take $\varepsilon$ to be small enough. Thus we can find $\xi$ with desired property.
\end{proof}

\begin{lemma} \label{L-to lower dimension}
Let $F$ be a nonempty finite subset of $G$ and $\delta>0$.
Let $\sigma$ be a map from $G$ to $\Sym(d)$ for some $d\in \Nb$.
Let $\Psi$ be a continuous map from $\Phi_d(\Map(\rho, F, \delta, \sigma))$ to $[0, 1]^{[d]\times \cU}$. Suppose that for any $\varphi \in \Map(\rho, F, \delta, \sigma)$ and
any $(a, U)\in [d]\times \cU$, if $\Phi_d(\varphi)_{a, U}=0$ or $1$, then $\Psi(\Phi_d(\varphi))_{a, U}=0$ or $1$ accordingly.
Then $\Psi\circ \Phi_d$
is $\cU^d|_{\Map(\rho, F, \delta, \sigma)}$-compatible.
\end{lemma}
\begin{proof}
Let $\varphi \in \Map(\rho, F, \delta, \sigma)$ and $a\in [d]$. By the choice of $\{f_U\}_{U\in \cU}$ there exists some $U\in \cU$ such that $\Phi_d(\varphi)_{a, U}=f_U(\varphi(a))=1$. By our assumption on $\Psi$ we then have $\Psi(\Phi_d(\varphi))_{a, U}=1$.

Let $\zeta\in \Psi(\Phi_d(\Map(\rho, F, \delta, \sigma)))$ and $a\in [d]$. It suffices to show that
there exists some $U\in \cU$ such that one has $\varphi(a)\in U$ for every $\varphi \in \Map(\rho, F, \delta, \sigma)$ satisfying
$\Psi(\Phi_d(\varphi))=\zeta$.
By the above paragraph we can find some $U\in \cU$ such that $\zeta_{a, U}=1$.
Let $\varphi\in \Map(\rho, F, \delta, \sigma)$ with $\Psi(\Phi_d(\varphi))=\zeta$. By our assumption on $\Psi$ we have $f_U(\varphi(a))=\Phi_d(\varphi)_{a, U}>0$.
Since $f_U$ vanishes on $X\setminus U$, we get $\varphi(a)\in U$.
\end{proof}

\begin{lemma} \label{L-maps down}
Let $W$ be a finite set and $Z$ a closed subset of $[0, 1]^W$. Let $m\in \Nb$ and $\xi\in (0, 1)^W$ such that for every $S\subseteq W$ with $|S|\ge m$ one has
$\xi|_S\not \in Z|_S$. Then there exists a continuous map $\Psi$ from $Z$ into $[0, 1]^W$ such that $\dim(\Psi(Z))\le m$ and  for any $z \in Z$ and
any $w\in W$, if $z_w=0$ or $1$, then $\Psi(z)_w=0$ or $1$ accordingly.
\end{lemma}
\begin{proof} This can be proved as  in the proof of \cite[Theorem 4.2]{LW}. We give a slightly different proof.

For each $S\subseteq W$, denote by $Y_S$ the subset of $[0, 1]^W$ consisting of elements whose coordinate at any $w\in W\setminus S$ is either $0$ or $1$. Then $Y_S$ is a closed subset of $[0, 1]^W$ with dimension $|S|$.
Set $Y=\bigcup_{|S|\le m}Y_S$. Since the union of finitely many closed subsets of dimension at most $m$ has dimension at most $m$ \cite[page 30 and Theorem V.8]{HW}, $Y$ has dimension at most $m$.

As $Z$ is compact, by the assumption on $\xi$ we can find $\delta>0$ such that $0<\xi_w-\delta<\xi_w+\delta<1$ for every $w\in W$ and the set
of $w\in W$ satisfying $\xi_w-\delta\le z_w\le \xi_w+\delta$ has cardinality at most $m$ for every $z\in Z$. For each $w\in W$, take a continuous map
$f_w:[0, 1]\rightarrow [0, 1]$ sending $[0, \xi_w-\delta]$ to $0$, and $[\xi_w+\delta, 1]$ to $1$.

Define a continuous map
$\Psi$ from $[0, 1]^W$ into itself by setting the coordinate of $\Psi(x)$ at $w\in W$ to be $f_w(x_w)$. For any $x\in [0, 1]^W$ and $w\in W$,
if $x_w=0$ or $1$, then clearly $\Psi(z)_w=0$ or $1$ accordingly. From the choice of $\delta$ and $f_w$ it is also clear that $\Psi(Z)\subseteq Y$.
Thus $\dim(\Psi(Z))\le \dim(Y)\le m$.
\end{proof}

We are ready to prove Theorem~\ref{T-top vs metric}.

\begin{proof}[Proof of Theorem~\ref{T-top vs metric}]
Let $\cU$ be a finite open cover of $X$, $\rho$ be a compatible metric on $X$, and $\theta>0$. Then we have $f_U$ as in Lemma~\ref{L-Lipschitz partion of unity}. For each $i\in \Nb$, define a  continuous map $\Phi_{d_i}: X^{[d_i]}\rightarrow [0, 1]^{[d_i]\times \cU}$
by $\Phi_{d_i}(\varphi)_{a, U}=f_U(\varphi(a))$ for $a\in [d_i]$ and $U\in \cU$.
Take $F, \delta, M, i$ and $\xi$ as in Lemma~\ref{L-small image}. By Lemmas~\ref{L-maps down} and \ref{L-to lower dimension} we can find a continuous map
$\Psi: \Phi_{d_i}(\Map(\rho, F, \delta, \sigma))\rightarrow [0, 1]^{[d_i]\times \cU}$ such that $\Psi\circ \Phi_{d_i}$ is $\cU^{d_i}|_{\Map(\rho, F, \delta, \sigma_i)}$-compatible and $\dim(\Psi(\Phi_{d_i}(\Map(\rho, F, \delta, \sigma))))\le (\mdim_{\Sigma, \rM}(X, \rho)+\theta)d_i$. From Lemma~\ref{L-compatible map}
we get $\cD(\cU, \rho, F, \delta, \delta_i)\le (\mdim_{\Sigma, \rM}(X, \rho)+\theta)d_i$. It follows that
$$ \cD_\Sigma(\cU)=\cD_\Sigma(\cU, \rho)\le \cD_\Sigma(\cU, \rho, F, \delta)\le \mdim_{\Sigma, \rM}(X, \rho)+\theta.$$
Since $\cU$ is an arbitrary finite open cover of $X$ and $\theta$ is an arbitrary positive number, we get $\mdim_\Sigma(X)\le \mdim_{\Sigma, \rM}(X, \rho)$.
As $\rho$ is an arbitrary compatible metric on $X$, we get $\mdim_\Sigma(X)\le \mdim_{\Sigma, \rM}(X)$.
\end{proof}

The Pontrjagin-Schnirelmann theorem \cite{PS} \cite[page 80]{Nagata}
says that for any compact metrizable space $Z$, the  dimension $\dim Z$ of $Z$ is equal to the minimal value of $\underline{\dim}_B(Z, \rho)$ for $\rho$ ranging over compatible metrics on $Z$. Since $\mdim_\Sigma(X)$ and $\mdim_{\Sigma, \rM}(X, \rho)$ are dynamical analogues of  $\dim(X)$ and $\underline{\dim}_B(X, \rho)$ respectively, it is natural to ask

\begin{question} \label{Q-top vs metric}
Let a countably infinite sofic group $G$ act continuously on a compact metrizable space $X$. Then is there any compatible metric $\rho$ on $X$ satisfying
$$ \mdim_\Sigma(X)=\mdim_{\Sigma, \rM}(X, \rho)?$$
\end{question}

Question~\ref{Q-top vs metric} was answered affirmatively by Lindenstrauss in the case $G=\Zb$ and $X$ has a nontrivial minimal factor \cite[Theorem 4.3]{Lindenstrauss}.

\section{Bernoulli Shifts} \label{S-shifts}

In this section we discuss sofic mean dimension of the Bernoulli shifts and their factors, proving Theorems~\ref{T-cube} and \ref{T-factor has positive dim}. Throughout this section $G$ will be a countable sofic group, and
$\Sigma$ will be a fixed sofic approximation sequence of $G$.

For any compact metrizable space $Z$, we have the left shift action $\alpha$ of $G$ on $Z^G$ given by $(sx)_t=x_{s^{-1}t}$ for all $x\in Z^G$ and $s, t\in G$.

\begin{theorem} \label{T-cube}
Let $Z$ be a compact metrizable space, and consider the left shift action of a countable sofic group $G$ on $X=Z^G$. Then
$$ \mdim_\Sigma(X)\le \mdim_{\Sigma, \rM}(X)\le \dim(Z).$$
If furthermore $Z$ contains a copy of $[0, 1]^n$ for every natural number $n\le \dim(Z)$ (for example, $Z$ could be any polyhedron or the Hilbert cube), then
$$\mdim_\Sigma(X)=\mdim_{\Sigma, \rM}(X)=\dim(Z).$$
\end{theorem}

The amenable group case  of Theorem~\ref{T-cube} was proved by Lindenstrauss and Weiss \cite[Propositions 3.1 and 3.3]{LW}, and Coornaert and Krieger \cite[Corollaries 4.2 and 5.5]{CK}.

\begin{remark} \label{R-how to construct map}
In general it is not easy to construct elements of $\Map(\rho, F, \delta, \sigma)$ for sofic non-amenable groups. Theorem~\ref{T-cube} implies that there are plenty of such elements for Bernoulli shifts. When $G$ is residually finite (for example free groups) and $\Sigma$ comes from a sequence of finite quotient groups of $G$, each periodic point in $X$ with the period being the corresponding finite-index normal subgroup of $G$ gives rise to an element of $\Map(\rho, F, \delta, \sigma)$ \cite[Theorem 7.1]{KL11} \cite[Theorem 1.3]{BL}.
\end{remark}

Recall the lower box dimension recalled at the beginning of Section~\ref{S-sofic metric mean dim}. We need the following lemma.

\begin{lemma} \label{L-shift upper bound}
Let $Z$ be a compact metrizable space, and consider the left shift action $\alpha$ of $G$ on $X=Z^G$. Let $\rho$ be a compatible metric on $Z$. Define
$\rho'$ by $\rho'(x, y)=\rho(x_{e_G}, y_{e_G})$ for $x, y\in X$. Then $\rho'$ is a dynamically generating continuous pseudometric on $X$.
Furthermore,  for any $\varepsilon>0$ one has
$$ \log N_\varepsilon(Z, \rho)\le h_{\Sigma ,\infty}^\varepsilon (\rho' )\le \log N_{\varepsilon/2}(Z, \rho).$$
In particular,
$ \mdim_{\Sigma, \rM}(X, \rho')=\underline{\dim}_B(Z, \rho)$.
\end{lemma}
\begin{proof} Clearly $\rho'$ is a dynamically generating continuous pseudometric on $X$. Let $\varepsilon>0$.

We show first $h_{\Sigma ,\infty}^\varepsilon (\rho' )\ge \log N_\varepsilon(Z, \rho)$.
It suffices to show $h_{\Sigma ,\infty}^\varepsilon (\rho' ,F, \delta ) \ge \log N_\varepsilon(Z, \rho)$ for every
nonempty finite subset $F$ of $G$ and every $\delta>0$.
Take a $(\rho, \varepsilon)$-separated subset $Y$ of $Z$ with
$|Y|=N_\varepsilon(Z, \rho)$.

Let $\sigma$ be a map from $G$ to  $\Sym(d)$ from some $d\in \Nb$ which is a good enough sofic approximation for $G$ such
that $\sqrt{1-|\cQ|/d}\cdot \diam(Z, \rho)\le \delta $, where
$$\cQ=\{a\in [d]: \sigma_{e_G}\circ\sigma_s(a)=\sigma_s(a) \mbox{ for all } s\in F\}.$$
For each map $f: [d]\rightarrow Y$, we define a map $\varphi_f: [d]\rightarrow X$ by
$$ (\varphi_f(a))_t=f(\sigma_{t^{-1}}(a))$$
for all $a\in [d]$ and $t\in G$. Let $s\in F$. For any $a\in \cQ$,  one has
\begin{align*}
(\alpha_s\circ \varphi_f(a))_{e_G}=(\varphi_f(a))_{s^{-1}}=f(\sigma_{s}(a))=f(\sigma_{e_G}\circ \sigma_s(a))=(\varphi_f\circ \sigma_s(a))_{e_G},
\end{align*}
and hence $\rho'(\alpha_s\circ \varphi_f(a), \varphi_f\circ \sigma_s(a))=0$. Thus
\begin{align*}
\rho'_2(\alpha_s\circ \varphi_f, \varphi_f\circ \sigma_s)\le \sqrt{(1-|\cQ|/d)(\diam(Z, \rho))^2}\le \delta.
\end{align*}
Therefore $\varphi_f\in \Map (\rho' ,F,\delta ,\sigma )$. For any distinct maps $f, g: [d]\rightarrow Y$, say $f(a)\neq g(a)$ for some
$a\in [d]$, one has
\begin{align*}
\rho'_\infty(\varphi_f, \varphi_g)&\ge \rho'(\varphi_f(\sigma_{e_G}^{-1}(a)), \varphi_g(\sigma_{e_G}^{-1}(a)))\\
&=\rho((\varphi_f(\sigma_{e_G}^{-1}(a)))_{e_G}, (\varphi_g(\sigma_{e_G}^{-1}(a)))_{e_G})\\
&=\rho(f(a), g(a))\ge \varepsilon.
\end{align*}
Thus the set $\{\varphi_f: f\in Y^{[d]}\}$ is $(\rho'_\infty, \varepsilon)$-separated. Therefore
$$N_\varepsilon(\Map (\rho' ,F,\delta ,\sigma ), \rho'_\infty)\ge |Y|^d=(N_\varepsilon(Z, \rho))^d.$$
It follows that $h_{\Sigma ,\infty}^\varepsilon (\rho' ,F, \delta ) \ge \log N_\varepsilon(Z, \rho)$ as desired.

Next we show $h_{\Sigma ,\infty}^\varepsilon (\rho' )\le \log N_{\varepsilon/2}(Z, \rho)$. It suffices to show
$h_{\Sigma ,\infty}^\varepsilon (\rho' ,\{e_G\}, 1 )\le \log N_{\varepsilon/2}(Z, \rho)$. Take a maximal $(\rho, \varepsilon/2)$-separated subset $Y$
of $Z$. Then $|Y|\le N_{\varepsilon/2}(Z, \rho)$.

Let $\sigma$ be a map from $G$ to  $\Sym(d)$ from some $d\in \Nb$. Let $\sE$ be a $(\rho'_\infty, \varepsilon)$-separated subset of
$\Map (\rho' ,\{e_G\},1 ,\sigma )$ with $|\sE|=N_\varepsilon(\Map (\rho' ,\{e_G\},1 ,\sigma ), \rho'_\infty)$. For each $\varphi\in \sE$,
we find some map $f_\varphi: [d]\rightarrow Y$ such that
$$ \max_{a\in [d]}\rho((\varphi(a))_{e_G}, f_\varphi(a))<\varepsilon/2.$$
If $\varphi, \psi\in \sE$ and $f_\varphi=f_\psi$, then
$$ \rho'_\infty(\varphi, \psi)=\max_{a\in [d]}\rho((\varphi(a))_{e_G}, (\psi(a))_{e_G})<\varepsilon,$$
and hence $\varphi=\psi$, since $\sE$ is $(\rho'_\infty, \varepsilon)$-separated. Therefore
$$ N_\varepsilon(\Map (\rho' ,\{e_G\},1 ,\sigma ), \rho'_\infty)=|\sE|\le |Y|^d\le (N_{\varepsilon/2}(Z, \rho))^d.$$
It follows that $h_{\Sigma ,\infty}^\varepsilon (\rho' ,\{e_G\}, 1 )\le \log N_{\varepsilon/2}(Z, \rho)$ as desired.
\end{proof}

We also need the following version of Lebesgue's covering theorem:

\begin{lemma}\cite[Lemma 3.2]{LW} \cite[Theorem IV.2]{HW} \label{L-cover of cube}
Let $W$ be a finite set. Let $\cU$ be a finite open cover of $[0, 1]^W$ such that no element of $\cU$ intersects two opposing sides of $[0, 1]^W$.
Then
$$ \cD(\cU)\ge |W|.$$
\end{lemma}

We are ready to prove Theorem~\ref{T-cube}.

\begin{proof}[Proof of Theorem~\ref{T-cube}] By Theorem~\ref{T-top vs metric}, Lemma~\ref{L-shift upper bound}, Proposition~\ref{P-allow pseudometric} and the Pontrjagin-Shnirelmann theorem as recalled at the end of Section~\ref{S-comparison},  we have $\mdim_\Sigma(X)\le \mdim_{\Sigma, \rM}(X)\le \dim(Z)$.

Now assume that $Z$ contains a copy of $[0, 1]^n$ for every natural number $n\le \dim(Z)$.
It suffices to show $\mdim_\Sigma(X)\ge \dim(Z)$. In turn it suffices to show $\mdim_\Sigma(X)\ge n$ for every natural number $n\le \dim(Z)$.
Since $([0, 1]^n)^G$ is a closed $G$-invariant subset of $Z^G$, by Proposition~\ref{P-subspace metric} one has
$\mdim_\Sigma(X)\ge \mdim_\Sigma(([0, 1]^n)^G)$. Therefore it suffices to show
$\mdim_\Sigma(([0, 1]^n)^G)\ge n$.

Take a finite open cover $\cU$ of $[0, 1]^n$ such that no element of $\cU$ intersects two opposing sides of $[0, 1]^n$.
For each $d\in \Nb$, note that  no element of $\cU^d$ intersects two opposing sides of $([0, 1]^n)^{[d]}=[0, 1]^{dn}$, and hence by Lemma~\ref{L-cover of cube}
one has $\cD(\cU^d)\ge dn$.

Denote by
$\pi$ the map $([0,1]^n)^G\rightarrow [0, 1]^n$ sending $x$ to $x_{e_G}$. Then $\tilde{\cU}:=\pi^{-1}(\cU)$ is a finite open cover of $([0,1]^n)^G$. Let $\rho$ be a compatible metric on $([0,1]^n)^G$. It suffices to show $\cD_\Sigma(\tilde{\cU}, \rho, F, \delta)\ge n$ for every nonempty finite subset $F$ of $G$ and every $\delta>0$.

Take a nonempty finite subset $K$ of $G$ such that if $x, y\in ([0,1]^n)^G$ are equal on $K$, then $\rho(x, y)<\delta/2$.

Let $\sigma$ be a map from $G$ to $\Sym(d)$ for some $d\in \Nb$ which is a good enough sofic approximation for $G$ such that
$\delta^2/4+(1-|\cQ|/d)(|F|+1)(\diam(([0,1]^n)^G, \rho))^2\le \delta^2$, where
$$\cQ=\{a\in [d]: \sigma_{t^{-1}}\circ\sigma_s(a)=\sigma_{t^{-1}s}(a) \mbox{ for all } s\in F, t\in K \mbox{ and } \sigma_{e_G}(a)=a\}.$$
For each map $f: [d]\rightarrow [0, 1]^n$, we define a map $\varphi_f: [d]\rightarrow ([0, 1]^n)^G$ by
\begin{align} \label{E-cube}
 (\varphi_f(a))_t=f(\sigma_{t^{-1}}(a))
\end{align}
for all $a\in [d]$ and $t\in G\setminus \{e_G\}$, and
\begin{align} \label{E-cube2}
 (\varphi_f(a))_{e_G}=f(a)
\end{align}
for all $a\in [d]$.  Note that (\ref{E-cube}) holds  for all $a\in \cQ$ and $t\in G$.
Set $\cR=\cQ\cap \bigcap_{s\in F}\sigma_s^{-1}(\cQ)$.
For any $a\in \cR$, $s\in F$, and $t\in K$, since $a, \sigma_s(a)\in \cQ$, one has
\begin{align*}
(\alpha_s\circ \varphi_f(a))_t=(\varphi_f(a))_{s^{-1}t}=f(\sigma_{t^{-1}s}(a))=f(\sigma_{t^{-1}}\circ \sigma_s(a))=(\varphi_f\circ \sigma_s(a))_t,
\end{align*}
and hence $\rho(\alpha_s\circ \varphi_f(a), \varphi_f\circ \sigma_s(a))<\delta/2$ by the choice of $K$. Thus
\begin{align*}
\rho_2(\alpha_s\circ \varphi_f, \varphi_f\circ \sigma_s)&\le \sqrt{\delta^2/4+(1-|\cR|/d)(\diam(([0, 1]^n)^G, \rho))^2}\\
&\le \sqrt{\delta^2/4+(1-|\cQ|/d)(|F|+1)(\diam(([0, 1]^n)^G, \rho))^2}  \le \delta.
\end{align*}
Therefore $\varphi_f\in \Map (\rho ,F,\delta ,\sigma )$. The map $\Phi: ([0, 1]^n)^{[d]}\rightarrow \Map (\rho ,F,\delta ,\sigma )$
sending $f$ to $\varphi_f$ is clearly continuous. Note that $\Phi^{-1}(\tilde{\cU}^d|_{\Map(\rho, F, \delta, \sigma)})=\cU^d$ because of the equation \eqref{E-cube2}. Therefore $\cD_\Sigma(\tilde{\cU}, \rho, F, \delta,\sigma)\ge \cD(\cU^d)\ge dn$.
It follows
that $\cD_\Sigma(\tilde{\cU}, \rho, F, \delta)\ge n$ as desired.
\end{proof}

\begin{theorem} \label{T-factor has positive dim}
Let $Z$ be a path-connected compact metrizable space, and  consider the left shift action $\alpha$ of a countable sofic group $G$ on $X=Z^G$. For
any nontrivial factor $Y$ of $X$, one has
$$ \mdim_\Sigma(Y)>0.$$
\end{theorem}

The amenable group case of Theorem~\ref{T-factor has positive dim} was proved by Lindenstrauss and Weiss \cite[Theorem 3.6]{LW}. We adapt their argument to our situation. First we prove a lemma:

\begin{lemma} \label{L-factor}
Let $G$ act continuously on a compact metrizable space $X$. Let $Y$ be a factor of $X$ with factor map $\pi: X\rightarrow Y$. Let $\cU$ be a finite open cover of $Y$. Then
$$\cD_\Sigma(\pi^{-1}(\cU))\le \cD_\Sigma(\cU).$$
\end{lemma}
\begin{proof} Let $\rho$ and $\rho'$ be compatible metrics on $X$ and $Y$ respectively. Replacing $\rho$ by
$\rho+\pi^{-1}(\rho')$ if necessary, we may assume that $\rho(x, y)\ge \rho'(\pi(x), \pi(y))$ for all $x, y\in X$.

Let $F$ be  a nonempty finite subset of $X$ and $\delta>0$. Let $\sigma$ be a map from $G$ to $\Sym(d)$ for some $d\in \Nb$.
For any $\varphi \in \Map(\rho, F, \delta, \sigma)$, one has $\pi\circ \varphi \in \Map(\rho', F, \delta, \sigma)$. Thus we have
a continuous map $\Phi:  \Map(\rho, F, \delta, \sigma)\rightarrow \Map(\rho', F, \delta, \sigma)$ sending $\varphi$ to $\pi\circ \varphi$.
Furthermore, $\Phi^{-1}(\cU^d|_{\Map(\rho', F, \delta, \sigma)})=(\pi^{-1}(\cU))^d|_{\Map(\rho, F, \delta, \sigma)}$. Thus
$\cD_\Sigma(\pi^{-1}(\cU), \rho, F, \delta, \sigma)\le \cD_\Sigma(\cU, \rho', F, \delta, \sigma)$. It follows that
$\cD_\Sigma(\pi^{-1}(\cU), \rho, F, \delta)\le \cD_\Sigma(\cU, \rho', F, \delta)$. Since $F$ is an arbitrary nonempty finite subset of $G$ and $\delta$ is an arbitrary positive number, we get
$$\cD_\Sigma(\pi^{-1}(\cU))=\cD_\Sigma(\pi^{-1}(\cU), \rho)\le \cD_\Sigma(\cU, \rho')=\cD_\Sigma(\cU).$$
\end{proof}

We are ready to prove Theorem~\ref{T-factor has positive dim}.

\begin{proof}[Proof of Theorem~\ref{T-factor has positive dim}] Denote by $\pi$ the factor map $X\rightarrow Y$. Let $\cU=\{U, V\}$ be an open cover of $Y$ such that none of $U$ and $V$ is dense in $Y$. By Lemma~\ref{L-factor} it suffices to show that $\cD_\Sigma(\pi^{-1}(\cU))>0$. Take compatible metrics $\rho$
and $\rho'$ on $Z$ and $X$ respectively.

Note that none of $\pi^{-1}(U)$ and $\pi^{-1}(V)$ is dense in $X$. Take $x_U\in X\setminus \overline{\pi^{-1}(U)}$ and $x_V\in X\setminus \overline{\pi^{-1}(V)}$.  Then there exist a finite symmetric subset $K$ of $G$ containing $e_G$
such that if $x\in X$
coincides with $x_U$ (resp. $x_V$) on $K$,
then $x\not \in \pi^{-1}(U)$ (resp. $x\not \in \pi^{-1}(V)$).
Since $Z$ is path connected, for each $s\in K$ we can take a continuous map $\gamma_s: [0, 1]\rightarrow Z$ such that $\gamma_s(0)=(x_U)_s$ and
$\gamma_s(1)=(x_V)_s$.

Now it suffices to show $\cD_\Sigma(\pi^{-1}(\cU), \rho', F, \delta)\ge 1/(2|K^2|)$ for every finite subset $F$ of $G$ containing $K$ and every $\delta>0$.
Take a finite symmetric subset $K_1$ of $G$ such that if two points $x$ and $y$ of $X$ coincide on $K_1$, then $\rho'(x, y)<\delta/2$.

Let $\sigma$ be a map from $G$ to $\Sym(d)$ for some $d\in \Nb$ which is a good enough sofic approximation for $G$ such that
$\sqrt{\delta^2/4+(1-|\cQ|/d)(\diam(X, \rho'))^2}\le \delta$ and
$|\cQ|/d\ge 1/2$, where
\begin{align*}
\cQ=\{a\in [d]: \sigma_s\sigma_t(a)&=\sigma_{st}(a) \mbox{ for all } s, t\in F\cup K_1, \mbox{ and } \sigma_{e_G}(a)=a, \\
\mbox{ and } \sigma_s(a)&\neq \sigma_t(a) \mbox{ for all distinct } s, t\in K\}.
\end{align*}
Take a maximal subset $\cW$ of $\cQ$ subject to the condition that the sets $\sigma(K)a$ for $a\in \cW$ are pairwise disjoint.
Then $\cQ\subseteq \sigma(K^2)\cW$. Thus $|\cW|/d\ge |\cQ|/(|K^2|d)\ge 1/(2|K^2|)$.

Now we define a map $\Phi: [0, 1]^{\cW}\rightarrow \Map (\rho' ,F,\delta ,\sigma )$.
Fix $z_0\in Z$. Let $f\in [0, 1]^{\cW}$. We define $\tilde{f}\in Z^{[d]}$ by
$$ \tilde{f}_{\sigma(s)a}=\gamma_{s^{-1}}(f(a))$$
for $a\in \cW$, $s\in K$, and
$$ \tilde{f}_b=z_0$$
for $b\in [d]\setminus \sigma(K)\cW$.  Then we define $\varphi_{\tilde{f}}\in X^{[d]}$ by
$$ (\varphi_{\tilde{f}}(a))_s=\tilde{f}(\sigma_{s^{-1}}(a))$$
for $a\in [d]$ and $s\in G$.
Let $s\in F$. For any $a\in \cQ$ and $t\in K_1$,  one has
\begin{align*}
(\alpha_s\circ \varphi_{\tilde{f}}(a))_t=(\varphi_{\tilde{f}}(a))_{s^{-1}t}=\tilde{f}(\sigma_{t^{-1}s}(a))=\tilde{f}(\sigma_{t^{-1}}\circ \sigma_s(a))=(\varphi_{\tilde{f}}\circ \sigma_s(a))_t,
\end{align*}
and hence $\rho'(\alpha_s\circ \varphi_{\tilde{f}}(a), \varphi_{\tilde{f}}\circ \sigma_s(a))<\delta/2$ by the choice of $K_1$. Thus
\begin{align*}
\rho'_2(\alpha_s\circ \varphi_{\tilde{f}}, \varphi_{\tilde{f}}\circ \sigma_s)\le \sqrt{\delta^2/4+(1-|\cQ|/d)(\diam(X, \rho'))^2}\le \delta.
\end{align*}
Therefore $\varphi_{\tilde{f}}\in \Map (\rho' ,F,\delta ,\sigma )$. Set $\Phi(f)=\varphi_{\tilde{f}}$.  Clearly $\Phi$ is continuous.

Set $\cV=\Phi^{-1}((\pi^{-1}(\cU))^d|_{\Map (\rho' ,F,\delta ,\sigma )})$. Let $f\in [0, 1]^{\cW}$ and $a\in \cW$. If $f(a)=0$, then
$(\varphi_{\tilde{f}}(a))_s=\tilde{f}_{\sigma(s^{-1})a}=\gamma_s(0)=(x_U)_s$ for all $s\in K$, and hence $\varphi_{\tilde{f}}(a)\not \in \pi^{-1}(U)$ by the choice
of $K$. Similarly, if $f(a)=1$, then $\varphi_{\tilde{f}}(a)\not \in \pi^{-1}(V)$. Thus no element of $\cV$ intersects two opposing sides of $[0, 1]^{\cW}$. By Lemma~\ref{L-cover of cube} we conclude that $\cD(\pi^{-1}(\cU), \rho', F, \delta, \sigma)\ge \cD(\cV)\ge |\cW|\ge d/(2|K^2|)$.
It follows that $\cD_\Sigma(\pi^{-1}(\cU), \rho', F, \delta)\ge 1/(2|K^2|)$ as desired.
\end{proof}

\section{Small-boundary Property}

In this section we discuss the relation between the small-boundary property and non-positive sofic mean topological dimension.

We start with recalling the definitions of zero inductive dimensional compact metrizable spaces and actions with small-boundary property.

A compact metrizable space $Y$ is said to have {\it inductive dimension $0$} if for every $y\in Y$ and every  neighborhood $U$ of $Y$
there exists a neighborhood $V$ of $y$ contained in $U$ such that the boundary $\partial V$ of $V$ is empty \cite[Definition II.1]{HW}.
A compact metrizable space has inductive dimension $0$ if and only if it has covering dimension $0$ \cite[Theorem V.8]{HW}.

\begin{definition} \label{D-small}
Let a countable group $\Gamma$ act continuously on a compact metrizable space $X$. We denote by $M(X, \Gamma)$ the set of $\Gamma$-invariant Borel probability measures on $X$. We say that a closed subset $Z$ of $X$ is {\it small} if $\mu(Z)=0$ for all $\mu \in M(X, \Gamma)$. In particular, when $M(X, \Gamma)$ is empty, every closed subset of $X$ is small. We say that the action has the {\it small-boundary property} (SBP) if for every point $x\in X$ and every neighborhood $U$ of $x$, there is an open neighborhood $V\subseteq U$ of $x$ with small boundary.
\end{definition}

When $\Gamma$ is amenable, for any subset $Z$ of $X$ it is easy to check that the function $F\mapsto \max_{x\in X}\sum_{s\in F}1_Z(sx)$ defined on the set of nonempty finite subsets of $\Gamma$ satisfies the conditions of the Ornstein-Weiss lemma \cite[Theorem 6.1]{LW}. Thus $\frac{1}{|F|}\max_{x\in X}\sum_{s\in F}1_Z(sx)$ converges to some limit as $F$ becomes more and more left invariant. Shub and Weiss defined $Z$ to be {\it small} if this limit is $0$ \cite{SW}. It is proved in page 538 of \cite{SW} that when $\Gamma$ is amenable and $Z$ is closed, the definition of Shub and Weiss coincides with Definition~\ref{D-small}. The notion of the SBP
was introduced in \cite{Lindenstrauss} and \cite{LW}.

If $X$ has less than $2^{\aleph_0}$ ergodic $G$-invariant Borel probability measures, then the action has the SBP \cite{SW} \cite[page 18]{LW}.

When $\Gamma$ is amenable, Lindenstrauss and Weiss showed that actions with the SBP has zero mean topological dimension \cite[Theorem 5.4]{LW}. We extend their result to sofic case:

\begin{theorem} \label{T-SBP to zero mean dim}
Let a countable sofic group $G$ act continuously on a compact metrizable space $X$. Suppose that the action has the SBP. Let $\Sigma$ be a sofic approximation sequence for $G$. Then $\mdim_\Sigma(X)\le 0$.
\end{theorem}

Lindenstrauss showed that if a continuous action of $\Zb$ on a compact metrizable space has a nontrivial minimal factor and has zero mean topological dimension, then it has the SBP \cite[Theorem 6.2]{Lindenstrauss}. Gutman  showed that if a continuous action of $\Zb^d$ for $d\in \Nb$ on a compact metrizable space has a free
zero-dimensional factor and has zero  mean topological dimension, then it has the SBP \cite[Theorem 1.11.1]{Gutman}. It is not clear in what generality the converse of Theorem~\ref{T-SBP to zero mean dim}  is true (see the discussion on page $20$ of \cite{LW}).

We shall adapt the argument of Lindenstrauss and Weiss to our case, by replacing the set $\{(sx)_{s\in H}: H\in G\}\subseteq X^H$ for an approximately left invariant finite subset $H$ of $G$ in their proof with the set $\Map(\rho, F, \delta, \sigma)$.
First we need the following three lemmas.

\begin{lemma} \label{L-measure}
Let a countable group $\Gamma$ act continuously on a compact metrizable space $X$. Let $Z$ be a closed subset of $X$. Then for any $\varepsilon>0$ there is an open neighborhood $U$ of $Z$ with $\sup_{\mu \in M(X, \Gamma)}\mu(U)<\varepsilon+\sup_{\mu \in M(X, \Gamma)}\mu(Z)$.
\end{lemma}
\begin{proof} Suppose that this is not true. Denote by $\cS$ the set of all open neighborhoods of $Z$, partially ordered by reversed inclusion. Then $M(X, \Gamma)$ is nonempty, and for each  $U\in \cS$ there is some $\mu_U\in M(X, G)$ with $\mu_U(U)\ge \varepsilon+\sup_{\mu \in M(X, \Gamma)}\mu(Z)$.
Take a limit point
$\nu$ of this net $\{\mu_U\}_{U\in \cS}$ in the compact space $M(X, \Gamma)$. We claim that $\nu(Z)\ge \varepsilon+\sup_{\mu \in M(X, \Gamma)}\mu(Z)$, which is a contradiction.

To prove the claim, by the regularity of $\nu$, it suffices to show $\nu(V)\ge \varepsilon+\sup_{\mu \in M(X, \Gamma)}\mu(Z)$
for every $V\in \cS$.
Take  $U'\in \cS$ with $\overline{U'}\subseteq V$.
Take a continuous function $f: X\rightarrow [0, 1]$ such that $f=1$ on $\overline{U'}$ and $f=0$ on $X\setminus V$.
For any $U\in \cS$ satisfying $U\subseteq U'$, one has $\int_X f\, d\mu_U\ge \mu_U(U)\ge \varepsilon+\sup_{\mu \in M(X, \Gamma)}\mu(Z)$. It follows
that $\nu(V)\ge \int_X f\, d\nu\ge \varepsilon+\sup_{\mu \in M(X, \Gamma)}\mu(Z)$, as desired.
\end{proof}

\begin{lemma} \label{L-small to function}
Consider a continuous action of a countable group $\Gamma$ on a compact metrizable space $X$ with the SBP. Then for any finite open cover $\cU$ of $X$ and any $\varepsilon>0$ there is a partition of unity $\phi_j: X\rightarrow [0, 1]$ for $j=1, \dots, |\cU|$ subordinate to $\cU$ such that $\sup_{\mu\in M(X, \Gamma)}\mu(\overline{\bigcup_{j=1}^{|\cU|}\phi_j^{-1}(0, 1)})<\varepsilon$.
\end{lemma}
\begin{proof} For each $x\in X$, take an open neighborhood $V_x$ of $x$ with small boundary such that $\overline{V_x}$ is contained in some element of $\cU$.
Since $X$ is compact, we can cover $X$ by finitely many such $V_x$'s. Note that if two subsets $Y_1$ and $Y_2$ of $X$ have small boundary $\partial Y_1$ and $\partial Y_2$ respectively, then $\partial (Y_1\cup Y_2)\subseteq \partial Y_1\cup \partial Y_2$ is also small. Thus we may take the union of those chosen $V_x$'s whose  closures are contained in one element of $\cU$ to obtain an open cover $\cU'$ of $X$ such that each element of $\cU'$ has small boundary and there is a bijection $\varphi: \cU'\rightarrow \cU$ with $\overline{U'}\subseteq \varphi(U')$ for every $U'\in \cU'$.

By Lemma~\ref{L-measure}, for each $U'\in \cU'$ we can find an open neighborhood $U''$ of the boundary $\partial U'$ of $U'$ such that $\sup_{\mu \in M(X, \Gamma)}\mu(U'')<\varepsilon/|\cU|$. Replacing $U''$ by $U''\cap \varphi(U')$ if necessary, we may assume that $U''\subseteq \varphi(U')$. Take an open neighborhood $U'''$ of $\partial U'$ such that $\overline{U'''}\subseteq U''$. List the elements of $\cU'$ as $U_1', \dots, U_{|\cU|}'$. For each $1\le j\le |\cU|$, take a continuous function $\psi_j: X\rightarrow [0, 1]$ such that $\psi_j=1$ on $\overline{U_j'}$ and $\psi_j=0$ on $X\setminus (U_j'\cup U_j''')$. Now define $\phi_j$ for $1\le j\le |\cU|$ inductively as $\phi_1=\psi_1$, and $\phi_j=\min(\psi_j, 1-\sum_{i=1}^{j-1}\phi_i)$ for $2\le j\le |\cU|$.

We claim that $\phi_1, \dots, \phi_{|\cU|}$ is a partition of unity subordinate to $\cU$. One has $0\le \phi_1=\psi_1\le 1$. For each $2\le j\le |\cU|$, one also has $\phi_j\le 1-\sum_{i=1}^{j-1}\phi_i$. Thus $\sum_{i=1}^j\phi_i\le 1$ for every $1\le j\le |\cU|$. In particular, one get $\phi_j=\min(\psi_j, 1-\sum_{i=1}^{j-1}\phi_j)\ge 0$ for every $2\le j\le |\cU|$. For any $x\in X$, if $x\in U'_1$ then $\phi_1(x)=\psi_1(x)=1$, otherwise $x\in U'_j$ for some $2\le j\le |\cU|$ and thus $\phi_j(x)=\min(\psi_j(x), 1-\sum^{j-1}_{i=1}\phi_i(x))=1-\sum^{j-1}_{i=1}\phi_i(x)$. It follows that
$\sum^{|\cU|}_{j=1}\phi_j=1$. Note that for each $1\le j\le |\cU|$, the support of $\phi_j$ is contained in $\overline{U'_j}\cup \overline{U'''_j}$, which in turn is contained in $\varphi(U'_j)$. This proves our claim.

Now one has
$$\overline{\bigcup_{j=1}^{|\cU|}\phi_j^{-1}(0, 1)}\subseteq \overline{\bigcup_{j=1}^{|\cU|}\psi_j^{-1}(0, 1)}\subseteq \overline{\bigcup_{j=1}^{|\cU|}U_j'''}\subseteq \bigcup_{j=1}^{|\cU|}U_j'',$$
 and hence $\mu(\overline{\bigcup_{j=1}^{|\cU|}\phi_j^{-1}(0, 1)})\le \mu(\bigcup_{j=1}^{|\cU|}U_j'')<\varepsilon$ for every $\mu \in M(X, \Gamma)$.
\end{proof}

\begin{lemma} \label{L-measure to orbit}
Let a countable group $\Gamma$ act continuously on a compact metrizable space $X$. Let $\rho$ be a compatible metric on $X$.
Let $Z$ be a closed subset of $X$, and $\varepsilon>0$. Then there exist a nonempty finite subset $F$ of $\Gamma$ and $\delta>0$ such that, for any map $\sigma$ from $\Gamma$ to $\Sym(d)$ for some $d\in \Nb$, one has
$$ \frac{1}{d}\max_{\varphi \in \Map(\rho, F, \delta, \sigma)}\sum_{a\in [d]}1_Z(\varphi(a))<\varepsilon+\sup_{\mu\in M(X, \Gamma)}\mu(Z).$$
\end{lemma}
\begin{proof} Suppose that that for any nonempty finite subset $F$ of $\Gamma$ and any $\delta>0$, there are a map $\sigma_{F, \delta}$ from $\Gamma$ to $\Sym(d)$ for some $d\in \Nb$ and a $\varphi_{F, \delta}\in \Map(\rho, F, \delta, \sigma_{F, \delta})$ with
$$ \frac{1}{d}\sum_{a\in [d]}1_Z(\varphi_{F, \delta}(a))\ge \varepsilon+\sup_{\mu\in M(X, \Gamma)}\mu(Z).$$
Denote by $\mu_{F, \delta}$ the probability measure $\frac{1}{d}\sum_{a\in [d]}\delta_{\varphi_{F, \delta}(a)}$ on $X$. The set of all such $(F, \delta)$ is partially ordered by $(F, \delta)\ge (F', \delta)$ when $F\supseteq F', \delta\le \delta'$. Take a limit point $\nu$ of $\{\mu_{F, \delta}\}_{F, \delta}$ in the compact set of Borel probability measures on $X$. Then there is a net $\{(F_j, \delta_j)\}_{j\in J}$ of pairs such that $\mu_{F_j, \delta_j}\to \nu$ as $j\to \infty$ and for any $(F, \delta)$ one has $(F_j, \delta_j)\ge (F, \delta)$ for all sufficiently large $j\in J$.

We claim that $\nu$ is $G$-invariant. Let $g$ be a continuous $\Rb$-valued function on $X$. For each $\delta>0$, set $C_\delta=\max_{x, y\in X, \rho(x, y)\le \delta}|g(x)-g(y)|$. For any $j\in J$  and any $s\in F_j$, since $\varphi_{F_j, \delta_j}\in \Map(\rho, F_j, \delta_j, \sigma_{F_j, \delta_j})$,
we have $|\cW_s|\ge d(1-\delta_j)$, where
$$\cW_s=\{a\in [d]: \rho(\varphi_{F_j, \delta_j}(sa), s\varphi_{F_j, \delta_j}(a))\le \sqrt{\delta_j}\}.$$
Thus, for any $j\in J$ and $s\in F_j$, one has
\begin{align*}
|\mu_{F_j, \delta_j}(g)-(s\mu_{F_j, \delta_j})(g)|&=\frac{1}{d}|\sum_{a\in [d]}g(\varphi_{F_j, \delta_j}(a))-\sum_{a\in [d]}g(s\varphi_{F_j, \delta_j}(a))|\\
&=\frac{1}{d}|\sum_{a\in [d]}g(\varphi_{F_j, \delta_j}(sa))-\sum_{a\in [d]}g(s\varphi_{F_j, \delta_j}(a))|\\
&\le \frac{1}{d}\sum_{a\in [d]}|g(\varphi_{F_j, \delta_j}(sa))-g(s\varphi_{F_j, \delta_j}(a))|\\
&\le \frac{1}{d} (|\cW_s|C_{\sqrt{\delta_j}}+(d-|\cW_s|)2\|g\|_\infty)\\
&\le C_{\sqrt{\delta_j}}+2\delta_j\|g\|_\infty.
\end{align*}
Letting $j\to \infty$, we have $\mu_{F_j, \delta_j}(g) \to \nu(g)$ and $(s\mu_{F_j, \delta_j})(g)\to (s\nu)(g)$ for every $s\in G$, and $\delta_j, C_{\sqrt{\delta_j}}\to 0$. It follows that $\nu(g)=(s\nu)(g)$. This proves the claim.

Next we claim that $\nu(Z)\ge \varepsilon+\sup_{\mu\in M(X, \Gamma)}\mu(Z)$, which is a contradiction.

To prove the claim, by the regularity of $\nu$, it suffices to show $\nu(U)\ge \varepsilon+\sup_{\mu\in M(X, \Gamma)}\mu(Z)$ for every open neighborhood $U$ of $Z$.
Take a continuous function $f: X\rightarrow [0, 1]$ such that $f=1$ on $Z$ and $f=0$ on $X\setminus U$. For any $j\in J$ one has
\begin{align*}
\mu_{F_j, \delta_j}(f)\ge \frac{1}{d}\sum_{a\in [d]}1_Z(\varphi_{F_j, \delta_j}(a))\ge \varepsilon+\sup_{\mu\in M(X, \Gamma)}\mu(Z).
\end{align*}
Letting $j\to \infty$, we get $\nu(U)\ge \nu(f)\ge \varepsilon+\sup_{\mu\in M(X, \Gamma)}\mu(Z)$ as desired.
\end{proof}

We are ready to prove Theorem~\ref{T-SBP to zero mean dim}.

\begin{proof}[Proof of Theorem~\ref{T-SBP to zero mean dim}]
Fix a compatible metric $\rho$ on $X$.
Let $\cU$ be a finite open cover of $X$. Set $k=|\cU|$. Let $\varepsilon>0$. Take $\phi_1, \dots, \phi_k$ as in Lemma~\ref{L-small to function} for $\cU$ and $\varepsilon$. Set $Z=\overline{\bigcup_{j=1}^k\phi_j^{-1}(0, 1)}$. Then $\sup_{\mu\in M(X, G)}\mu(Z)\le \varepsilon$. By Lemma~\ref{L-measure to orbit} we can find a nonempty finite subset $F$ of $G$ and $\delta>0$ such that, for any map $\sigma$ from $G$ to $\Sym(d)$ for some $d\in \Nb$, one has
$$ \frac{1}{d}\max_{\varphi \in \Map(\rho, F, \delta, \sigma)}\sum_{a\in [d]}1_Z(\varphi(a))<\varepsilon+\sup_{\mu\in M(X, G)}\mu(Z)\le 2\varepsilon.$$

Define $\Phi: X\rightarrow \Rb^k$ by $\Phi(x)=(\phi_1(x), \dots, \phi_k(x))$. Define $\Phi_{F, \delta, \sigma}: \Map(\rho, F, \delta, \sigma)\rightarrow \Rb^{kd}$ by
$$ \Phi_{F, \delta, \sigma}(\varphi)=(\Phi(\varphi(1)), \dots, \Phi(\varphi(d))).$$

Let $e^i_j$, $i=1, \dots, d, j=1, \dots, k$ be the standard basis of $\Rb^{kd}$. For every $I\subseteq [d]$ with $|I|\le 2\varepsilon d$ and every
$\xi\in \{0, 1\}^{kd}$, define
$$C(I, \xi)={\rm span} \{e^i_j: i\in I, 1\le j\le k\}+\xi.$$
Then
$$ \Phi_{F, \delta, \sigma}(\Map(\rho, F, \delta, \sigma))\subseteq \bigcup_{|I|\le 2\varepsilon d, \xi\in \{0, 1\}^{kd}}C(I, \xi).$$
Note that $\Phi_{F, \delta, \sigma}$ is $\cU^d|_{\Map(\rho, F, \delta, \sigma)}$-compatible, and $\bigcup_{|I|\le 2\varepsilon d, \xi \in \{0, 1\}^{kd}}C(I, \xi)$ is a finite union of at most $2\varepsilon kd$ dimensional affine subspaces of $\Rb^{kd}$.
Since the union of finitely many closed subsets of dimension at most $2\varepsilon kd$ has dimension at most $2\varepsilon kd$ \cite[page 30 and Theorem V.8]{HW},
from Lemma~\ref{L-compatible map} we get
$$ \cD(\cU, \rho, F, \delta, \sigma)\le 2\varepsilon kd.$$
It follows that $\cD_\Sigma(\cU)\le 0$, and hence $\mdim_\Sigma(X)\le 0$.
\end{proof}

We leave the following

\begin{question} \label{Q-SBP and zero dimension}
Could one strengthen Theorem~\ref{T-SBP to zero mean dim} to get $\mdim_{\Sigma, \rM}(X)=0$?
\end{question}


\end{document}